\documentclass[10pt]{article}
\usepackage[utf8]{inputenc} 
\usepackage[english]{babel}

\usepackage{amsmath}
\usepackage{amsthm}
\usepackage{amssymb}
\usepackage{color}

\usepackage{imakeidx}

\usepackage{color}
\makeindex[intoc]

\usepackage[pagewise,mathlines]{lineno}
\usepackage{hyperref}
\usepackage{amsfonts}

\usepackage{mathtools}
\mathtoolsset{showonlyrefs}

\newcommand{\brc}[1]{  \left\{#1\right\} } 
  
\newcommand{\norm}[1]{  \left\|#1\right\| } 
\newcommand{\pare}[1]{\left(#1\right)}    
\newcommand{\core}[1]{  \left[#1\right] }      
\newcommand{\abs}[1]{  \left\vert#1\right\vert }     
     
\newcommand{\bb}[1]{\mathbb{#1}}   
 
\newcommand*{\R}{\mathbb R}

\newtheorem{lemma}{Lemma}[section]
\newtheorem{cor}{Corollary}[section]

\newtheorem{theorem}{Theorem}[section]

\newtheorem{remark}{Remark}[section]

\begin{document}
	
	\title{\bf Homogenization for nonlocal problems with smooth kernels}
	
	\author{Monia Capanna, Jean C. Nakasato, Marcone C. Pereira and
		Julio D. Rossi}
		
		\date{}
	
	\maketitle

	{Keywords: heterogeneous media, homogenization, nonlocal equations, Neumann problem, Dirichlet problem.}
	
	\begin{abstract} 
		In this paper we consider the homogenization problem for a nonlocal
		equation that involve different smooth kernels. We assume that the spacial domain 
		is divided into a sequence of two subdomains $A_n \cup B_n$ and we have three different smooth
		kernels, one that controls the jumps from $A_n$ to $A_n$, a second one that controls
		the jumps from $B_n$ to $B_n$ and the third one that governs the interactions between $A_n$ and $B_n$.
		Assuming that $\chi_{A_n} (x) \to X(x)$ weakly-* in $L^\infty$ (and then $\chi_{B_n} (x) \to (1-X)(x)$ weakly-* in $L^\infty$)
		as $n \to \infty$ 
		we show that there is an homogenized limit system in which the three kernels and the limit function $X$ appear. 
		We deal with both Neumann and Dirichlet boundary conditions. Moreover, we also provide
		a probabilistic interpretation of our results.  
	\end{abstract}

	\section{Introduction}
	\label{Sect.intro}

	Our main goal in this paper is to study the homogenization that occurs when one deals with nonlocal problems with 
	different non-singular kernels in that act in different domains.

	Consider a partition of the ambient space into two subdomains $A$, $B$, and consider a nonlocal problem
	in which we have three different smooth kernels. One (that we call $J$) that measure the probability of jumping
	from $A$ to $A$ ($J(x,y)$ is the probability that a particle that is at $x\in A$ moves to $y\in A$), 
	another one ($G$) that is involved in jumps from $B$ to $B$ and a third one ($R$)
	that gives the interactions between $A$ and $B$.  When we look at the stationary problem associated with this
	jump
	process with three different kernels, we obtain a nonlocal problem of the form
	\begin{equation*} \label{1.1.intro}
	\begin{array}{rl}
	\displaystyle & f(x) = \displaystyle \chi_{A} (x) \int_{A} J (x,y) (u (y) - u (x)) dy  
	+ \chi_{A} (x) \int_{B} R (x,y) (u (y) - u (x)) dy
	\\[10pt]
	& \displaystyle 
	\qquad \qquad + \chi_{B} (x) \int_{A} R (x,y) (u (y) - u (x)) dy 
	+ \chi_{B} (x) \int_{B} G (x,y) (u (y) - u (x)) dy.
	\end{array}
	\end{equation*}
	Here $\chi_{A}$ and $\chi_{B}$ denote the characteristic functions of $A$ and $B$ respectively.
	
	This equation can also be written as a system,
	\begin{equation*} \label{1.1.intro.sys}
	\begin{array}{ll}
	\displaystyle f(x) = \displaystyle  \int_{A} J (x,y) (u (y) - u (x)) dy 
	+  \int_{B} R (x,y) (u (y) - u (x)) dy, \qquad & x \in A, 
	\\[10pt]
	\displaystyle 
	f(x) =   \int_{A} R (x,y) (u (y) - u (x)) dy
	+ \int_{B} G (x,y) (u (y) - u (x)) dy, \quad  & x \in B.
	\end{array}
	\end{equation*}
	
	Remark that the involved kernels can be of convolution type, that is, we could have
	for instance, $J(x,y) = J(x-y)$ (this special form of the kernels is often used in applications). 
	However, we only use in our arguments that the kernels $V =J$, $G$ and $R$ 
	are non-singular functions which satisfy the following hypotheses that will be assumed from now on
	$$
	{\bf (H)} \qquad 
	\begin{gathered}
	V \in \mathcal{C}(\R^N\times \R^N,\R) \textrm{ is non-negative with } V(x,x)>0, \; V(x,y) = V(y,x) \textrm{ for every $x,y \in \R^N$, and } \\
	\int_{\R^N} V(x,y) \, dy = 1.
	\end{gathered}
	$$

	Now, in this setting, we consider a sequence of partitions, $A_n$, $B_n$, of a fixed domain $\Omega$,
	$$
	\Omega = A_n \cup B_n,
	$$
	 and assume that we have weak convergence of the characteristic functions of $A_n$ and $B_n$ as $n \to \infty$,
	\begin{equation} \label{convergencia}
	\chi_{A_n} (x) \to X(x) \quad \textrm{ weakly-* in } L^\infty (\Omega),
	\end{equation}
	and consequently, we have 
	$$
	\chi_{B_n} (x) \to (\chi_\Omega - X)(x) \quad \textrm{ weakly-* in } L^\infty (\Omega).
	$$
	
	Our main goal is to pass to the limit as $n \to \infty$ in the sequence of solutions, $u_n$, to the nonlocal 
	problem associated with the partition $A_n$, $B_n$. We consider both the homogeneous Neumann and Dirichlet problems. 
	Our main results can be summarized as follows (see Section \ref{sect-Main-results} for more precise statements):
	
	\begin{theorem}[Main Theorem]
	{\it For any partition $A_n$, $B_n$, there exists a unique solution to our nonlocal problem. Assuming 
	 \eqref{convergencia}, 
	there is a weak limit of $u_n$ in $L^2$ as $n\to \infty$ and this limit is characterized as being
	the sum of the two components of an associated limit system for which we have uniqueness. 
	Moreover, we show that there
	is a corrector type result, that is, if we add an extra term (the corrector) to the sequence $u_n$ we 
	obtain strong convergence in $L^2$. 
	}
	\end{theorem}
	
	For this limit homogenization procedure we have an interpretation in probabilistic terms analyzing the evolution of a particle that moves in 
	$\overline{\Omega}$, see Section \ref{sect-Main-results}.
	
Now let us end with brief description of previous results and comment on the ideas and difficulties involved in our proofs. 

Nonlocal equations with smooth kernels like the ones considered here has been widely studied and used in the literature 
as models in different applied scenarios, see for example, \cite{ElLibro,BCh,CF,Chasseigne-Chaves-Rossi-2006,crew,delia1,delia2,delia3,F,Gal}.
Here we have a model in which the jumping probabilities depends on three different kernels $J$, $G$ and $R$
that act in different parts of the domain (thus, our model problem can be seen as a coupling between nonlocal equations in $A$ and $B$).
For other couplings (even considering local equations and nonlocal ones, we refer to \cite{CaRo,delia1,delia2,delia3,Du,GQR,Gal,Kra}.

Homogenization for PDEs is by now a classical subject that originated in the study of the behaviour
of the solutions to elliptic and parabolic local equations with highly oscillatory coefficients (periodic homogenization). 
We refer to
\cite{BLP,TT, Tar} as general references for the subject.
For other kinds of homogenization for pure nonlocal problems with one kernel we
refer to \cite{marc1,marc2,marc3}. For homogenization results for equations with a singular 
kernel we refer to \cite{Ca,sw2,Wa} and references therein. We emphasize that those references
deal with homogenization in the coefficients involved in the equation. For random homogenization of an obstacle problem we refer to \cite{caffa2}. 
For mixing local and nonlocal processes we refer to \cite{CaRo}.
Here we deal with an homogenization problem that is different in nature with the ones treated in the previously mentioned references
as we homogenize mixing three different jump operators with smooth kernels. 

%Notice that in this homogenization procedure
%there is a nonlocal system instead of a single equation
%(the sum of $u_A$ and $u_B$ does not verify a single equation). This is due to the fact that
%the original problem can be written as a system. 

Finally, let us describe the main ingredients that appear in the proofs. First, we show 
weak convergence along subsequences of $u_n$, $\chi_{A_n}u_n$ and $\chi_{B_n}u_n$ (these convergences comes from
a uniform bound in $L^2$). Next, we find the system that these limits verify. This part of the proof
is delicate since we have to  
pass to the limit in the eek form of the equation that involves terms like $\chi_{B_n} (x) \chi_{A_n} (y) u_n (y) J(x-y)$ and we only have weak convergence of $\chi_{B_n}$ and
$\chi_{A_n} u_n$. 
Here we need to rely on the continuity of $J$ and use the fact that the product
$\chi_{B_n} (x) \chi_{A_n} (y) u_n(y) J(x-y)$ involves two different variables, $x$ and $y$. Finally, we show
uniqueness of the limit by proving uniqueness of solutions to the limit system.

	\medskip
	
	The paper is organized as follows: in Section \ref{sect-Main-results} we state our main
	results; in Section \ref{sect-esto} we provide a probabilistic interpretation to both, Neumann and Dirichlet problems; Section \ref{Sect.proofs} we 
	deal with the Neumann case and, finally, in Section \ref{Sect.proofs.Dir} we consider the Dirichlet problem.

	\section{Statements of the main results} \label{sect-Main-results}

	\subsection{The Neumann problem} We first deal with Neumann boundary conditions.
	The domain $\Omega$ is divided in two subdomains $A_n$, $B_n$ depending on a parameter $n$,
	$$
	\Omega = A_n \cup B_n, \qquad A_n \cap B_n = \emptyset.
	$$
	Assume that \eqref{convergencia} holds.
	For each $n$,
	we consider,
	\begin{equation} \label{1.1}
	\begin{array}{rl}
	\displaystyle & f(x) = \displaystyle \chi_{A_n} (x) \int_{A_n} J (x,y) (u_n (y) - u_n (x)) dy  
	+ \chi_{A_n} (x) \int_{B_n} R (x,y) (u_n (y) - u_n (x)) dy
	\\[10pt]
	& \displaystyle 
	\qquad \qquad + \chi_{B_n} (x) \int_{A_n} R (x,y) (u_n (y) - u_n (x)) dy 
	+ \chi_{B_n} (x) \int_{B_n} G (x,y) (u_n (y) - u_n (x)) dy,
	\end{array}
	\end{equation}
	for $x \in \Omega$.
	
	Associated with this nonlocal problem we have an energy
	\begin{equation} \label{1.1.energy}
	\begin{array}{rl}
	\displaystyle E_n(u) = & \displaystyle \frac14 \int_{A_n}  \int_{A_n} J (x,y) (u (y) - u (x))^2 dy \, dx  
	+ \frac14 \int_{B_n}  \int_{B_n} G (x,y) (u (y) - u (x))^2 dy \, dx 
	\\[10pt]
	& \displaystyle 
	+ \frac12 \int_{A_n}  \int_{B_n} R (x,y) (u (y) - u (x))^2 dy \, dx - \int_\Omega f (x) u (x) \, dx .
	\end{array}
	\end{equation}
	
	Existence and uniqueness of the solutions to \eqref{1.1} which is also a minimizer of \eqref{1.1.energy} in 
	\begin{equation}\label{SetW_solutions}
	W= \left\{u \in L^2 (\Omega) \ : \ \int_\Omega u =0 \right\}
	\end{equation}
	is shown in Theorem \ref{EUN} assuming $f \in W$. Next, we deal with
	the homogenization procedure and study the limit as $n\to \infty$.

	\begin{theorem} \label{theo1.intro} 
		Let $\{u_n\} \subset W$ be a family of solutions of \eqref{1.1}. 
		Then, there exists a unique $(u_A, u_B) \in L^2(\Omega) \times L^2(\Omega)$ with $\int_\Omega \{ u_A(x) + u_B(x) \} dx =0$ and  
		\begin{equation*}
		\begin{gathered}
		\chi_{A_n}u_n\rightharpoonup u_A \quad \textrm{ and } \quad 
		\chi_{B_n}u_n\rightharpoonup u_B \quad \textrm{ weakly in } L^2(\Omega).
		\end{gathered}
		\end{equation*}
		The pair $(u_A, u_B)$ is characterized as being the unique solution to the following system 
		\begin{equation} \label{limitAn}
		\begin{array}{rl}
		X(x) f(x) = & \displaystyle\int_{\Omega}J(x,y)\left[X(x)u_A(y)-X(y)u_A(x) \right]dy \\[10pt]
		& \displaystyle 
		+  \int_\Omega R (x,y) \left[X(x)u_B (y)-u_A(x)(1-X(y)) \right] dy ,
		\end{array}
		\end{equation}
		and 
		\begin{equation}  \label{limitBn}
		\begin{array}{rl}
		(1-X(x)) f(x) = 
		& \displaystyle\int_\Omega  R (x,y) \left[u_A (y)(1-X(x)) -u_B(x)X(y) \right]dy
		\\[10pt]
		& + \displaystyle\int_\Omega G (x,y)\left[ u_B (y)(1-X(x))-u_B(x)(1-X(y))\right]  dy.
		\end{array}
		\end{equation}
	\end{theorem}

	Notice that in this homogenization procedure
there is a nonlocal system instead of a single equation
(the sum of $u_A$ and $u_B$ does not verify a single equation). This is due to the fact that
the original problem can be written as a system. 

As a consequence of the fact that $(u_A, u_B)$ is the unique solution to the \eqref{limitAn}, 
we have that $u_A$ and $u_B$ also satisfy the equation
		\begin{equation} \label{limf}
		\begin{array}{rl}
		f(x) = & \displaystyle\int_{\Omega}J(x,y)\left[X(x)u_A(y)-X(y)u_A(x) \right]dy
		+\displaystyle\int_\Omega G(x,y)\left[ (1-X(x))u_B(y)- (1-X(y))u_B(x) \right] dy \\[10pt]
		& +\displaystyle\int_\Omega R(x,y)\left[(1-X(x))u_A(y)-(1-X(y))u_A(x)+X(x)u_B(y)-X(y)u_B(x)\right] dy. 
		\end{array}
		\end{equation}

Let us now consider the extreme cases $X(x)=0$ or $X(x)=1$ in $\Omega$. 
In this case, the limit problem is given by a scalar equation defined for just one kernel. 
Also, we obtain strong convergence of the solutions.
	\begin{cor} \label{corX=0}
		Assume $X(x)=0$ a.e. $x\in\R^n$. 
		Then, the first component of the limit system is zero,
		$$
		u_A(x)=0\textrm{ a.e. } x \in \Omega,
		$$
		we have strong convergence
		$$u_n \to u_B \qquad \mbox{ strongly in $L^2(\Omega)$,} $$ 
		and the homogenized equation is given by
		\begin{equation*}
		f(x)=\int_\Omega G(x,y)\left( u_B(y)- u_B(x) \right) dy.
		\end{equation*}
	\end{cor}

		Notice that an analogous result can be shown assuming $X(x) =1$ in $\Omega$.  
	\begin{cor} \label{corX=1}
		Assume $X(x)=1$ a.e. $x\in\Omega$. 
		Then
		$$
		u_B(x)=0\textrm{ a.e. } x\in\Omega,
		$$
		$$u_n \to u_A \qquad \mbox{strongly in $L^2(\Omega)$,}$$ and the homogenized equation is given by
		\begin{equation*}
		 f(x)=\int_\Omega J(x,y)\left( u_A(y)- u_A(x) \right) dy.
		\end{equation*}
	\end{cor}

Both cases, $X(x)=0$ or $X(x)=1$, can be seen as a kind of performing weakly perforations in $\Omega$ setting a prevailing situation at the limit equation, see \cite{marc2}.

	Finally, let us give a corrector result for the solutions of the Neumann problem \eqref{1.1}. In order to do that, we need to assume $1>X(x)>0$ uniformly in $\Omega$. Let
$$
\omega_n (x) = \dfrac{\chi_{A_n}(x) u_A(x)}{X (x)} + \dfrac{\chi_{B_n}(x)u_B(x)}{1-X(x)}.
$$
Then, we have the following theorem.

		\begin{theorem}\label{Corrector.Neumann} Let us assume the there exist two constants $c_0>0$ and $c_1>0$ such that 
		\begin{equation} \label{conX}
		1-c_1 \geq X(x) \geq c_0>0, \quad \forall x \in \Omega.
		\end{equation}
			Then
			\begin{equation} \label{n630}
			\left\|u_n - \left(\omega_n - \int_\Omega \omega_n \, dx \right) \right\|_{L^2(\Omega)}\to 0 \quad \textrm{ as } n \to +\infty.
			\end{equation}
		\end{theorem}

In this result, we have the flavor of the classical corrector approach introduced in \cite{BLP}.

	\subsection{The Dirichlet problem} 
	Now we divide the whole $\mathbb{R}^N$ in two disjoint subdomains $A_n$, $B_n$ depending on a parameter $n$,
	$$
	\mathbb{R}^N = A_n \cup B_n, \qquad A_n \cap B_n = \emptyset.
	$$
	
	Let $\chi_{A_n}$ and $\chi_{B_n}$ be the characteristic functions of $A_n$ and $B_n$ respectively. 
	We will assume 
	\begin{equation*} \label{convergencia_dir}
	\chi_{A_n} (x) \to X \quad \textrm{ weakly-* in } L^\infty (\mathbb{R}^N).
	\end{equation*}
	Consequently,  
	$$
	\chi_{B_n} (x) \to1 - X \quad \textrm{ weakly-* in } L^\infty (\mathbb{R}^N).
	$$
	
	As before, we assume that we have three different kernels
	$J$, $G$ and $R$. 
	The kernel $J$ controls the jumps from $A_n$ to $A_n$, $J(x-y)$ is the probability that a particle that is at $x\in A_n$ moves to $y\in A_n$;
	$G$ from $B_n$ to $B_n$ and finally $R$ controls the jumps from $A_n$ to $B_n$.
	
	{\bf The Dirichlet problem.}
	For $x \in \Omega$ we have
	\begin{equation} \label{1.1.Dir}
	\begin{array}{rl}
	\displaystyle f(x) = & \displaystyle \chi_{A_n} (x) \int_{A_n} J (x,y) (u_n (y) - u_n (x)) dy  
	+ \chi_{A_n} (x) \int_{B_n} R (x,y) (u_n (y) - u_n (x)) dy
	\\[10pt]
	& \displaystyle 
	+ \chi_{B_n} (x) \int_{A_n} R (x,y) (u_n (y) - u_n (x)) dy
	+ \chi_{B_n} (x) \int_{B_n} G (x,y) (u_n (y) - u_n (x)) dy
	\end{array}
	\;  x \in \Omega
	\end{equation}
	and we impose that
	\begin{equation*} \label{1.1.fuera}
	u (x)  \equiv 0, \qquad x \not\in \Omega.
	\end{equation*}

	Associated with this nonlocal problem we have an energy for functions in the space
	\begin{equation}\label{SetW_solutions_Dir}
	W_{Dir} = \Big\{u \in L^2 (\R^N) \ : \ u =0\quad\textrm{in}\quad \mathbb{R}^N \setminus \Omega \Big\}
	\end{equation}
	given by
	\begin{equation} \label{1.1.energy.Dir}
	\begin{array}{rl}
	\displaystyle E_n(u) = & \displaystyle \frac14 \int_{A_n}  \int_{A_n} J (x,y) (u (y) - u (x))^2 dy \, dx 
	 + \frac14 \int_{B_n}  \int_{B_n} G (x,y) (u (y) - u (x))^2 dy \, dx \\[10pt]
	& \displaystyle 
	+ \frac12 \int_{A_n}  \int_{B_n} R (x,y) (u (y) - u (x))^2 dy \, dx - \int_\Omega f (x) u (x) \, dx .
	\end{array}
	\end{equation}
	
	Existence and uniqueness of the solution to \eqref{1.1.Dir} which is obtained as the unique minimizer of \eqref{1.1.energy.Dir} in 
	$
	W_{Dir}$
	is shown in Theorem \ref{EUSD} assuming $f \in L^2(\Omega)$ (there is no need to assume that the mean value of $f$ is zero here).
	
	\begin{theorem} \label{theo1.intro.dir}
		Let $\{u_n\}$ be a family of solutions of \eqref{1.1.Dir}. 
		Then, there exists a unique pair $(u_A, u_B) \in L^2(\Omega) \times L^2(\Omega)$ with $u_A(x) = u_B(x)\equiv0$ for $x\in\R^N\backslash\Omega$ and  
		\begin{equation*}
		\begin{gathered}
		\chi_{A_n}u_n\rightharpoonup u_A \quad \textrm{ and } \quad 
		\chi_{B_n}u_n\rightharpoonup u_B \quad \textrm{ weakly in } L^2(\Omega),
		\end{gathered}
		\end{equation*}
		such that, for $x\in \Omega$, $u_A$ and $u_B$ satisfy the following system 
		\begin{equation} \label{limitAn_dir}
		\begin{array}{rl}
		X(x) f(x) = & \displaystyle\int_{\R^N}J(x,y)\left[X(x)u_A(y)-X(y)u_A(x) \right]dy \\[10pt]
		& \displaystyle 
		+  \int_{\R^N} R (x,y) \left[X(x)u_B (y)-u_A(x)(1-X(y)) \right] dy ,
		\end{array}
		\end{equation}
		and 
		\begin{equation}  \label{limitBn_dir}
		\begin{array}{rl}
		(1-X(x)) f(x) = 
		& \displaystyle\int_{\R^N}  R (x,y) \left[u_A (y)(1-X(x)) -u_B(x)X(y) \right]dy
		\\[10pt]
		& + \displaystyle\int_{\R^N} G (x,y)\left[ u_B (y)(1-X(x))-u_B(x)(1-X(y))\right]  dy.
		\end{array}
		\end{equation}
	\end{theorem}

		As a consequence, we have that $u_A$ and $u_B$ also satisfy
		\begin{equation*} \label{limf_dir}
		\begin{array}{rl}
		f(x) = & \displaystyle\int_{\R^N}J(x,y)\left[X(x)u_A(y)-X(y)u_A(x) \right]dy
		+\displaystyle\int_{\R^N} G(x,y)\left[ (1-X(x))u_B(y)- (1-X(y))u_B(x) \right] dy \\[10pt]
		& +\displaystyle\int_{\R^N} R(x,y)\left[(1-X(x))u_A(y)-(1-X(y))u_A(x) +X(x)u_B(y)-X(y)u_B(x)\right] dy .
		\end{array}
		\end{equation*}

	As in the Neumann problem, we obtain from the homogenization procedure a nonlocal system instead of a single equation
since the sum of $u_A$ and $u_B$ does not verify a single one. 

Also, one can notice that the homogenized systems for Neumann and Dirichlet problems present quite similar expressions. They differ by the integration domain and space where the solutions belong.

Next, let us consider a particular partition to the whole $\R^N$. We assume that $\chi_{A_n}\to 1$ as $n\to \infty$. 
Thus, we are assuming $X(x) \equiv 1$ in $\Omega$. The other case $\Omega \subset B_n$ for all $n$ (i.e. $X(x) \equiv 0$ in $\Omega$) is analogous and is left to the reader. 
			\begin{cor} \label{corExD}
				Suppose that 
				$$
				\chi_{A_n}\to 1 \quad \mbox{ as } n\to \infty.
				$$
				Then, $u_n\to u_A$ strongly in $L^2(\Omega)$ and the homogenized equation is given by
				\begin{equation*}
				\begin{array}{rl}
				f(x)=&\displaystyle \int_{\R^N}J(x,y)( u_A(y) - u_A(x) ) dy .
				\end{array}
				\end{equation*}				
			\end{cor}

We finish this section mentioning the validity of a corrector result for the Dirichlet problem \eqref{1.1.Dir} under the same conditions assumed in 
Theorem \ref{Corrector.Neumann}. 

			\begin{theorem} \label{Corrector.Dirichlet}
				Let us assume the there exist $c_0$ and $c_1>0$ such that 
				\begin{equation*} \label{conX.Dir}
				1-c_1 \geq X(x) \geq c_0>0, \quad \forall x \in \Omega,
				\end{equation*}
				and let us set  
				$$
				\psi_n (x) = \dfrac{\chi_{A_n}(x)u_A(x)}{X(x)} + \dfrac{\chi_{B_n}(x)u_B(x)}{1-X(x)}.
				$$
				Then,
				\begin{equation*} \label{n630.Dir}
				\Big\|u_n - \psi_n \Big\|_{L^2(\Omega)}\to 0 \quad \textrm{ as } n \to +\infty.
				\end{equation*}
			\end{theorem}

\section{The stochastic model} \label{sect-esto}

In this section we provide the details of a probabilistic interpretation of our original model
and the obtained homogenization results. To be in a probabilistic context we assume that
\begin{align*}
\int_{\Omega}J(x,y)dy= 1, \quad \int_{\Omega}R(x,y)dy = 1,\quad \int_{\Omega}G(x,y)dy=1, \qquad\forall x\in\overline\Omega.
\end{align*}

\subsection{The Neumann problem}
We want to analyze the evolution of a particle that moves in $ \overline\Omega$. To describe
the movements of the particle we introduce three 
families $\brc{E^1_k}_{k\in \bb N}$, $\brc{E^2_k}_{k\in \bb N}$ and $\brc{E^3_k}_{k\in \bb N}$ of independent 
random variables with exponential distribution of parameter $1/3$ (we think about them as having three independents clocks). Define
\begin{align*}
\Upsilon_k:=\min_{i\in\brc{1,2,3}}\brc{E_k^i},\qquad\forall k\in\mathbb N.
\end{align*} 
The set $\brc{\Upsilon_k}_{k\in\mathbb N}$ is a family of independent random variables distributed as an exponential of parameter $1$.
Fixing $\tau_0=0$, we define recursively the random times
\begin{align}\label{times}
\tau_k=\tau_{k-1}+\Upsilon_k, \qquad \forall k\in \mathbb N.
\end{align}
We denote by $Y_n\pare{t}$ the position of the particle at time $t$.
The evolution of the particle is described as follows. At the times $\{\tau_k\}$ the particle chooses a site $y\in {\Omega}$ according to the kernels $J$, $R$ or $G$ (that is given by the clock that rings, that is, the random variable that realizes the minimum in $\min_{i\in\brc{1,2,3}}\brc{E_k^i}$). The jumps from a site in $A_n$ to another site in $A_n$ are ruled by $J$,  the jumps between $A_n$ and $B_n$ (or vice versa) are ruled by $R$, the jumps from a site in $B_n$ to a site in $B_n$ are ruled by $G$. More precisely, if $\Upsilon_k=E_k^1$ the particle chooses a site $y\in \Omega$ according to $J\pare{Y_n(\tau_{k-1}),y}$ and it jumps on it only if $Y_n(\tau_{k-1})\in A_n$ and $y\in A_n$ otherwise the particle remains in its current position. If $\Upsilon_k=E_k^2$ the particle chooses a site $y\in \Omega$ according to the kernel $R\pare{Y_n(\tau_{k-1}),y}$ and it jumps on it only if $Y_n(\tau_{k-1})\in A_n$ (or in $B_n$) and $y\in B_n$ (or in $A_n$ respectively). Finally if $\Upsilon_k=E_k^3$ the particle chooses a site $y\in \Omega$ according to $G\pare{Y_n(\tau_{k-1}),y}$ and it jumps on it only if $Y_n(\tau_{k-1})\in B_n$ and $y\in B_n$.

The process $Y_n(t)$ is a Markov process whose generator ${L}_n$ is defined on functions  $f\in C\pare{A_n}\cap C\pare{B_n}$ as
\begin{equation}\label{gen.intro}
\begin{array}{ll}\displaystyle 
{L}_n f(x)=
& \displaystyle  \chi_{A_n}\pare{x}\int_{\Omega}\chi_{A_n}\pare{y} J\pare{x,y}\pare{f\pare{y}-f\pare{x}}dy+\chi_{B_n}\pare{x}\int_{\Omega}\chi_{B_n}\pare{y} G\pare{x,y}\pare{f\pare{y}-f\pare{x}}dy\\[10pt]
& \displaystyle  +\chi_{A_n}\pare{x}\int_\Omega \chi_{B_n}\pare{y}R\pare{x,y}\pare{f\pare{y}-f\pare{x}}dy+\chi_{B_n}\pare{x}\int_\Omega \chi_{A_n}\pare{y}R\pare{x,y}\pare{f\pare{y}-f\pare{x}}dy.
\end{array}
\end{equation}

In the next lemma we give an explicit expression of the solution of our Neumann problem \eqref{1.1} in terms of the process $\pare{Y_n(t)}_t$. Given $x\in\overline\Omega$, in what follows we denote by $P^x$ the probability measure defined on the canonical probability space of the process starting at $x$ at time $t=0$. This means that
\begin{align*}
P^x\pare{Y_n(0)=x}=1.
\end{align*} 
Moreover we will use $\mathbb {E}^x\core{\cdot}$ to the denote the expectation respect to the probability $P^x$.
\begin{lemma} \label{lema-Neu.prob}
Let $f\in W$ and $u_n$ be the solution of the Neumann problem \eqref{1.1}. Then
\begin{align*}\label{solnp}
u_n(x)=-\int_0^\infty \mathbb E^x\core{f\pare{Y_n(t)}}dt,\qquad x\in\overline\Omega.
\end{align*}
\end{lemma}

\begin{proof}
Consider the evolution problem associated with the generator $L_n$ defined in 
(\ref{gen.intro}), with initial datum $v_0 \in W\cap L^\infty\pare{\overline\Omega}$ and right hand side $f(x) \in W$,
\begin{equation*}\label{evol.intro}
\left\{
\begin{array}{ll}
\displaystyle \frac{\partial v_n}{\partial t} (x,t) = {L}_n v_n(x,t) - f(x), \qquad & x\in \overline\Omega, \, t>0, \\[10pt]
v_n(x,0)=v_0(x), \qquad & x\in \overline\Omega.
\end{array}
\right.
\end{equation*}
By Theorem 2.15 in \cite{lig} we can write
\begin{equation*}\label{funv}
v_n (x,t) = S_n(t) v_0(x) - \int_0^t S_n(t-s) f(x) \, ds,
\end{equation*}
where $S_n(\cdot)$ is the semigroup associated to the generator $L_n$. 

Since $S_n(t)f(x)=\mathbb {E}^x\core{f\pare{Y_n(t)}}$,  we can interpret $v_n(x,t)$ as the expected total amount that is collected assuming that $f(x)$ is a running
payoff and that $v_0$ is the initial payoff. That is,
\begin{align}\label{expv}
v_n(x,t) = \mathbb{E}^x \core{ v_0 \pare{Y_n(t)}}-\int_0^t \mathbb E^x\core{f(Y_n(s)) }\, ds. 
\end{align}
%In the last step we used that, by Fubini's Theorem, $\displaystyle\int_0^t\mathbb E^x\pare{ f(Y_n(s))} \, ds=\bb E^x\pare{\int_0^t f(Y_n(s)) \, ds}$.
The solution $u_n$ of the Neumann problem \eqref{1.1} is given by
$$
u_n (x)= \lim_{t\to \infty } v_n (x,t)
$$
and therefore, by \eqref{expv}, we get that
\begin{align*}
u_n(x)&=\lim_{t\to\infty}\brc{\mathbb{E}^x \core{ v_0 (Y_n(t))}-\int_0^t\mathbb E^x\core{ f(Y_n(s)} \, ds }\\
&=\lim_{t\to\infty}\mathbb{E}^x \core{v_0 (Y_n(t))}-\int_0^\infty\mathbb E^x\core{f(Y_n(s))} \, ds .
\end{align*}
To conclude the proof it is enough to show that
\begin{align}\label{indconvzero}
\lim_{t\to\infty}\mathbb{E}^x \core{ v_0 (Y_n(t)) }=0,
\end{align}
independently on the initial condition $v_0\in W\cap L^\infty\pare{\overline\Omega}$.
From now on, we use  $\mathcal M_1\pare{\overline\Omega}$ to denote the space of probability measures on $\overline\Omega$ and, for every $\mu\in\mathcal M_1\pare{\overline\Omega}$, we use $\mu S(t)$ to denote the measure in $\mathcal M_1\pare{\overline\Omega}$ such that
\begin{align*}
\int_{\overline\Omega}f\;d\pare{\mu S(t)}=\int_{\overline\Omega} S(t)f\;d\mu, \qquad\forall f\in C\pare{A_n}\cap C\pare{B_n}.
\end{align*}
To prove \eqref{indconvzero} we show that
\begin{itemize}
\item the process $\pare{Y_n(t)}_t$ has a unique invariant measure given by $$\nu^*\pare{dx}=\frac{dx}{\abs{\overline\Omega}},$$
\item for every $n\in\mathbb N$ there exists $$\lim_{t\to\infty}\delta_x S_n(t),$$
where $\delta_x$ is the Dirach measure in $\overline\Omega$ centered at $x$.
\end{itemize}

Indeed, once we show the two conditions above, by Proposition 1.8 in \cite{lig} we can conclude that 
$$Y_n(t)\xrightarrow[t\to\infty]{D}Y_n(\infty),$$ 
where $Y_n(\infty)$ is distributed according to the invariant measure $\nu*$ and therefore 
$$\lim_{t\to\infty}\mathbb E^x\pare{v_0(Y_n(t))}=\mathbb E^x\pare{v_0(Y_n(\infty))}=\int_{\overline\Omega}v_0(x)\nu^*(dx)=\frac{1}{\abs{\overline\Omega}}\int_{\overline\Omega} v_0(x)dx=0, \qquad\forall v_0\in W\cap L^\infty\pare{\overline\Omega}.$$
In the previous step we used that, by dominated convergence Theorem, we can write $$\lim_{t\to\infty}\mathbb E^x\pare{v_0(Y_n(t))}=\mathbb E^x\pare{v_0(Y_n(\infty))}.$$ 
Limit \eqref{indconvzero} is then proved.

We prove now the first of the two conditions. By Proposition 1.8 in \cite{lig} we  know that the set of the invariant measures of the process $(Y_n(t))_t$ is given by

$$\mathcal I_n=\brc{\nu\in \mathcal M_1(\overline\Omega): \int_{\overline\Omega}L_nfd\nu=0, \, \forall f\in C\pare{A_n}\cap C\pare{B_n}}.$$

The measure $\nu^*(dx):=\displaystyle \frac{dx}{\abs{\overline\Omega}}\in\mathcal I_n$ because $\nu^*$ is a solution of
\begin{equation}\label{eqinvm}
\int_{\overline\Omega}L_nfd\nu=0, \qquad\forall f\in  C\pare{A_n}\cap C\pare{B_n}.
\end{equation}
To conclude that $\nu^*$ is the unique probability measure solution of \eqref{eqinvm} we need to show that, if $\nu$ is such that \eqref{eqinvm} holds, then there exists $k\in\mathbb R$ such that $\nu=k\nu^*$. Let $\nu$ be a solution of \eqref{eqinvm} and consider $F:C\pare{A_n}\cap C\pare{B_n}\to\mathbb R$ and $G:C\pare{A_n}\cap C\pare{B_n}\to\mathbb R$, the two linear operators defined as
\begin{align*}
F(g):=\frac{1}{|\overline\Omega|}\int_{\overline\Omega} g(x)dx, \quad G(g):=\int_{\overline\Omega} g(x)\nu(dx),\quad\forall  g\in C\pare{A_n}\cap C\pare{B_n}.
\end{align*}
Fix $h\in C\pare{A_n}\cap C\pare{B_n}$ such that $\displaystyle\int_{\overline\Omega}h(x)dx=0$. There exists a unique $f_{n, h}\in C\pare{A_n}\cap C\pare{B_n}$ solution to 
\begin{align*}
L_nf_{n, h}(x)=h(x)
\end{align*}
with $\displaystyle\int_{\overline\Omega}h(x)d\nu=0$.  This implies that $F(h)=0$ and also $G(h)=0$, hence
$\ker(F)\subseteq \ker (G)$ and consequently there exists $k\in \mathbb R$ such that $G(g)=kF(g)$, for all $g\in C\pare{A_n}\cap C\pare{B_n}$. This allows to conclude that $$\nu(dx)=\frac{k}{\abs{\overline\Omega}}dx.$$

We prove now the second condition. Since $\overline\Omega$ is compact, the set $\mathcal M_1\pare{\overline\Omega}$ endowed with the weak topology is also compact (see Theorem 3.4 in \cite{Var}).
Therefore the family $\pare{\delta_x S_n(t)}_{t\in [0, \infty)}$ is relatively compact in $\mathcal M_1(\overline \Omega)$ and, by the first condition and Proposition 1.8 in \cite{lig}, we can conclude that any sequence $\pare{\delta_x S_n(t_k)}_{k\in\mathbb N}$ is convergent and
\begin{align*}
\lim_{t_k\to\infty}\delta_x S_n(t_k)=\nu^*.
\end{align*}
This is enough to conclude that
\begin{align*}
\lim_{t\to\infty}\delta_x S_n(t)=\nu^*
\end{align*}
and therefore our second condition holds.
\end{proof}

From Lemma \ref{lema-Neu.prob} we have that the solution $u_n$ to the Neumann problem \eqref{1.1}
can be seen as
$$
u_n(x)=-\int_0^\infty \mathbb E^x\core{f\pare{Y_n(t)}}dt,\qquad x\in\overline\Omega.
$$

Now, our result, Theorem \ref{theo1.intro}, says that the limit as $n \to \infty$ of $u_n$ is given by the sum of the components of the pair
$(u_A, u_B) \in L^2(\Omega) \times L^2(\Omega)$ that is a solution to the limit system \eqref{limitAn}--\eqref{limitBn}.
Therefore, we have obtained the limit as $n \to \infty$ of the expected value of $f$ along the trajectories of the process $Y_n$
starting at $x$,
$$
-\int_0^\infty \mathbb E^x\core{f\pare{Y_n(t)}}dt,
$$
as
$$
\lim_{n \to \infty} u_n(x)= \lim_{n \to \infty} -\int_0^\infty\mathbb E^x\core{ f\pare{Y_n(t)}}dt
= u_A (x) + u_B (x) ,\qquad x\in\overline\Omega.
$$

\subsubsection{The Dirichlet problem}

Concerning the Dirichlet problem, the movement of the particle obeys the same rules as before,
but now the particle is allowed to jump outside $\overline\Omega$, and, as soon as this happens, the particle is
killed and disappears from the system.
In this new model we denote by $Z_n(t)$ the position of the particle in $\overline\Omega$ and we suppose 
(as we did before, but this time in the whole $\mathbb{R}^N$) that we have probability kernels in our equations, that is,
\begin{align*}
\int_{\mathbb {R}^N}J(x,y)dy= 1, \quad \int_{\mathbb {R}^N}R(x,y)dy= 1,\quad \int_{\mathbb {R}^N}G(x,y)dy= 1, \qquad\forall x\in\overline\Omega.
\end{align*}
The process $\pare{Z_n(t)}_{t\geq 0}$ is a Markov process whose generator ${L}_n$ is defined on functions  $f\in C\pare{A_n}\cap C\pare{B_n}$ such that $\text{supp} f \subseteq \overline\Omega$ as
\begin{align}\label{gen.intro2}
\begin{aligned}
\hspace{-25pt}{L}_n f(x)&=\chi_{\brc{A_n\cap\overline\Omega}}\pare{x}\int_{\mathbb {R}^N}\chi_{A_n}\pare{y} J\pare{x,y}\pare{f\pare{y}-f\pare{x}}dy+\chi_{\brc{B_n\cap\overline\Omega}}\pare{x}\int_{\mathbb {R}^N}\chi_{B_n}\pare{y} G\pare{x,y}\pare{f\pare{y}-f\pare{x}}dy\\
&\quad +\chi_{\brc{A_n\cap\overline\Omega}}\pare{x}\int_{\mathbb {R}^N} \chi_{B_n}\pare{y}R\pare{x,y}\pare{f\pare{y}-f\pare{x}}dy+\chi_{\brc{B_n\cap\overline\Omega}}\pare{x}\int_{\mathbb {R}^N} \chi_{A_n}\pare{y}R\pare{x,y}\pare{f\pare{y}-f\pare{x}}dy.
\end{aligned}
\end{align}

The probability that the particle situated in a site $x\in\overline\Omega$ is killed in one jump is given by 
\begin{align}\label{probout}
\begin{aligned}
q_n(x):=
&\chi_{\brc{A_n\cap\overline\Omega}}\pare{x}\pare{\int_{\overline\Omega^c}\chi_{A_n}\pare{y} J\pare{x,y}dy+\int_{\overline\Omega^c}\chi_{B_n}\pare{y} R\pare{x,y}dy}\\
&+\chi_{\brc{B_n\cap\overline\Omega}}\pare{x}\pare{\int_{\overline\Omega^c}\chi_{A_n}\pare{y} R\pare{x,y}dy+\int_{\overline\Omega^c}\chi_{B_n}\pare{y} G\pare{x,y}dy}.
\end{aligned}
\end{align}
In what follows, to simplify the exposition, we suppose that
\begin{align}\label{condpos}
q_n^{\text{inf}}:=\inf_{x\in\overline\Omega}q_n(x)>0,
\end{align}
and we comment in Remark \ref{rem.K} how to deal with the general case in which the particle
has a uniformly positive probability of being killed in less than a finite number of jumps.

In the next lemma we give the probabilistic interpretation of the solution to the Dirichlet problem \eqref{1.1.Dir} in terms of the process $\pare{Z_n(t)}_t$.
\begin{lemma}
Let $f\in L^\infty(\overline\Omega)$ and $u_n$ be the solution of \eqref{1.1.Dir}. Let $s_n:=\inf\brc{t\geq 0: Z_n(t)\not\in \overline\Omega}$ be the first time at which the process jumps outside $\overline\Omega$. Then, 
\begin{align}\label{stoppingtime}
P^x\pare{s_n <\infty}=1,\qquad x\in \overline\Omega,
\end{align}
and
\begin{align}\label{soldir}
u_n(x)=-\mathbb E^x\core{\int_0^{s_n} f\pare{Z_n(t)}dt}, \qquad x\in\overline\Omega.
\end{align}
\end{lemma}
\begin{proof}
%We start by proving \eqref{stoppingtime}. 
%Let $\pare{\tau_k}_{k\in\mathbb N}$ be the sequence of times defined  in \eqref{times}.  Observe that
%\begin{align*}
%\brc{S_n<\infty}=\brc{\exists k\in\mathbb N: Z^n\pare{\tau_k}\not\in\Omega}.
%\end{align*}
%Therefore to conclude \eqref{stoppingtime} it is sufficient to show that
%\begin{align}\label{stoppingtime2}
%P^x\pare{Z^n\pare{\tau_k}\in\Omega,\;\forall k\in\mathbb N}=0,\qquad\forall x\in \Omega.
%\end{align}
%Let $x\in\Omega$. Observe that
%\begin{align*}
%\begin{aligned}
%P^x\pare{Z^n\pare{\tau_k}\in\Omega,\;\forall k\in\mathbb N}&=P^x\pare{\cap_{k=1}^\infty\brc{Z^n\pare{\tau_k}\in\Omega}}\\
%&\leq P^x\pare{\cap_{k=1}^N\brc{Z^n\pare{\tau_k}\in\Omega}},
%\end{aligned}
%\end{align*}
%for all $N\in\mathbb N$. Therefore to conclude \eqref{stoppingtime2} it is enough to show that 
%\begin{align}\label{stoppingtime3}
%P^x\pare{\cap_{k=1}^N\brc{Z^n\pare{\tau_k}\in\Omega}}\leq q^N,
%\end{align}
%for some $q\in (0,1)$. Let analyse the case in which $N=2$. Observe that
%\begin{align*}
%P^x\pare{\cap_{k=1}^2\brc{Z_n\pare{\tau_k}\in\Omega}}&=P^x\pare{Z_n\pare{\tau_2}\in\Omega\big|Z_n\pare{\tau_1}\in\Omega}P^x\pare{Z_n\pare{\tau_1}\in\Omega}\\
%&\leq \pare{1-q_n^{\text{inf}}}^2,
%\end{align*}
%where $q_n^{\text{inf}}$ is defined in \eqref{condpos}. In a similar way we can prove that
%\begin{align*}
%P^x\pare{\cap_{k=1}^N\brc{Z^n\pare{\tau_k}\in\Omega}}\leq \pare{1-q_n^{\text{inf}}}^N,\quad\forall N\in\mathbb N.
%\end{align*}
%By condition \eqref{condpos} we conclude \eqref{stoppingtime3}. 
We prove \eqref{stoppingtime} by showing the stronger property
\begin{align}\label{finexp}
\mathbb E^x\core{s_n}<\infty.
\end{align}
Observe that
\begin{align}\label{exp1}
\displaystyle\mathbb{E}^x\core{s_n}&=\int_0^\infty P^x\pare{s_n>t}dt=\int_0^\infty\int_{\overline\Omega}\nu_t^n\pare{dy}dt,
\end{align}
where $\nu_t^n\pare{dy}$ is the probability measure on $\mathbb R^N$ %$\in\mathcal M_1\pare{\mathbb R}$
such that
\begin{align}
P^x\pare{Z_n(t)\in A}=\int_A\nu_t^n(dy),
\end{align}
for every measurable $A\subseteq \mathbb R^N$. By Lemma A.5.1 in \cite{KL}, the probability measure $\nu_t^n\pare{dy}$ satisfies
\begin{align}\label{dendir}
\begin{cases}
\displaystyle\frac{d}{dt}\int_{\overline\Omega}g(y)\nu_t^n\pare{dy}=\int_{\mathbb R^N}L_n g(y)\nu_t^n\pare{dy}, \qquad & t>0,\\[10pt]
\displaystyle\int_{\overline\Omega^c}g(y)\nu_t^n\pare{dy}=0,\qquad & t>0,\\[10pt]
\displaystyle\int_{\mathbb R^N} g(y)\nu_0^n\pare{dy}=g(x),
\end{cases}
\end{align}
for every $g\in C(A_n)\cap C(B_n)$ such that $\text{supp} g\subseteq\overline\Omega$.
Recall the expression of $L_n$ given in \eqref{gen.intro2} and take $g(x)=\mathbb \chi_{\overline\Omega}(x)$ in \eqref{dendir}. We obtain that
\begin{align}\label{dernu}
\begin{aligned}
\displaystyle\frac{d}{dt}\int_{\overline\Omega}\nu_t^n\pare{dy}=\int_{\mathbb R^N}
&\Big\{-\chi_{\brc{A_n\cap\overline\Omega}}\pare{y}\pare{\int_{\overline\Omega^c}\chi_{A_n}\pare{z} J\pare{y,z}dz
+\int_{\overline\Omega^c}\chi_{B_n}\pare{z} R\pare{y,z}dz}\\
\displaystyle& \qquad -\chi_{\brc{B_n\cap\overline\Omega}}\pare{y}\pare{\int_{\overline\Omega^c}\chi_{A_n}\pare{z} R\pare{y,z}dz+\int_{\overline\Omega^c}\chi_{B_n}\pare{z} G\pare{y,z}dz}\Big\}\nu_t^n(dy).
\end{aligned}
\end{align}
In the previous step we used that since $\text{supp } g\subseteq \overline\Omega$ and $J, R$ and $G$ are symmetric
\begin{align}
&\int_{\mathbb R^N}\chi_{\brc{D\cap\overline\Omega}}\pare{y}\pare{\int_{\mathbb {R}^N}\chi_{E}\pare{z} V\pare{y,z}\pare{g\pare{z}-g\pare{y}}dz}d\nu_t^n(dy)\\
&\hspace{+20pt}=\int_{\overline\Omega}\int_{\overline\Omega}\chi_{D}\pare{y}\chi_{E}(z)V\pare{y,z}\pare{g\pare{z}-g\pare{y}}dz\,d\nu_t^n(dy)-\int_{\overline\Omega}\chi_{D}\pare{y}g(y)\pare{\int_{\overline\Omega^c}\chi_{E}\pare{z}V(y, z)dz}d\nu_t^n(dy)\\
&\hspace{+20pt}=-\int_{\mathbb R^N}\chi_{\brc{D\cap \overline\Omega}}\pare{y}g(y)\pare{\int_{\overline\Omega^c}\chi_{E}\pare{z}V(y, z)dz}d\nu_t^n(dy),
\end{align}
for every $V\in\brc{J, R, G}$ and $D, E\in \brc{A_n, B_n}$.

Recall the definition of $q_n(\cdot)$ and $q_n^{\text{inf}}$ given in \eqref{probout} and \eqref{condpos} respectively, by \eqref{dernu} we can write that
\begin{align*}
\frac{d}{dt}\int_{\overline\Omega}\nu_t^n\pare{dy}=\int_{\overline\Omega}-q_n(y)\nu_t^n(dy)\leq -q_n^{\text{inf}}\int_{\overline\Omega}\nu_t^n(dy).
\end{align*}
Therefore 
\begin{align}\label{ine1}
\int_{\overline\Omega}\nu_t^n(dy)\leq e^{-q_n^{\text{inf}}t}\int_{\overline\Omega}\nu_0^n(dy)=e^{-q_n^{\text{inf}}t}.
\end{align}
By \eqref{exp1} and \eqref{ine1} we get that
\begin{align}\label{bound_exp_st}
\mathbb{E}^x\pare{s_n}\leq\int_0^\infty e^{-q_n^{\text{inf}}t}dt=\frac{1}{q_n^{\text{inf}}}.
\end{align}
This concludes the proof of \eqref{finexp}.

We prove now \eqref{soldir}. If $f\in L^\infty(\overline\Omega)$ also $u_n\in L^\infty(\overline\Omega)$. Therefore, by Lemma A.5.1 of \cite{KL} we know that
\begin{align}\label{dyn}
M_n\pare{t}:=u_n\pare{Z_n\pare{t}}-u_n\pare{Z_n\pare{0}}-\int_0^tL_nu_n\pare{Z_n\pare{s}}ds
\end{align}
is a martingale with respect to the natural filtration generated by the process. 
Once we prove that $\pare{M_n\pare{t\wedge s_n}}_{t\geq 0}$  is uniformly integrable, that is
\begin{align}\label{unifint}
\lim_{K\to \infty}\sup_{t\in [0, \infty)}\int_{\brc{\abs{M_n(t\wedge s_n)}>K}}\abs{M_n(t\wedge s_n)}dP^x=0,
\end{align}
by the Doob's Optional Stopping Theorem (see \cite{DW}) we obtain that $\mathbb E^x\core{M_n\pare{s_n}}=\mathbb E^x\core{M_n\pare{0}}=0$ and consequently
\begin{align}\label{rel2}
\begin{aligned}
\mathbb E^x\core{\int_0^{s_n}L_nu_n\pare{Z_n\pare{s}}ds }&=\mathbb E^x\core{Z_n\pare{s_n}}-u_n(x)\\
&=-u_n(x).
\end{aligned}
\end{align}
Since $u_n$ is the solution of \eqref{1.1.Dir}, by \eqref{rel2}, we get \eqref{soldir}.
Therefore to conclude it remains just to show \eqref{unifint}.
By \eqref{dyn} we get that
\begin{align}\label{bound_stopped_mart}
\abs{M_n(t\wedge s_n)}\leq 2\norm{u_n}_\infty+8s_n\norm{u_n}_\infty,\qquad\forall t\geq 0.
\end{align}
Therefore, for every $t\geq 0$ and $K\geq 0$, it holds the following inclusion of events
\begin{align}\label{incl_ev}
\brc{\abs{M_n(t\wedge s_n)}>K}\subseteq \brc{s_n>\frac{K-2\norm{u_n}_\infty}{8\norm{u_n}_\infty}}.
\end{align}
To simplify the notation we set $$\mathcal A_{n, K}:=\brc{s_n>\frac{K-2\norm{u_n}_\infty}{8\norm{u_n}_\infty}}.$$
From \eqref{bound_stopped_mart} and \eqref{incl_ev} we get that
\begin{align}\label{fin_ine}
\begin{aligned}
\sup_{t\in [0, \infty)}\int_{\brc{\abs{M_n(t\wedge s_n)}>K}}\abs{M_n(t\wedge s_n)}dP^x&\leq \sup_{t\in [0, \infty)}\int_{\brc{\abs{M_n(t\wedge s_n)}>K}}\pare{2\norm{u_n}_\infty+8s_n\norm{u_n}_\infty}dP^x\\
&\leq 2\norm{u_n}_\infty P^x\pare{\mathcal A_{n, K}}+8\norm{u_n}_\infty\int_{\mathcal A_{n, K}}s_ndP^x.
\end{aligned}
\end{align}
By \eqref{finexp} we can conclude that the limit of the right hand side of \eqref{fin_ine} converges to $0$ as $K\to\infty$. Therefore \eqref{unifint} is proved.
\end{proof}

As for the Neumann case, we have that our homogenization result for $u_n$ given in Theorem \ref{theo1.intro.dir}
(the limit as $n\to \infty$ of $u_n$ is given by the sum of the components of the pair $(u_A, u_B) \in L^2(\Omega) \times L^2(\Omega)$ 
that solves \eqref{limitAn_dir}--\eqref{limitBn_dir}) 
implies the limit of the expected value of $f$ along the trajectories of the process $Z_n$ until the particle jumps off $\Omega$ (and
gets killed),
$$
u_n(x)=-\mathbb E^x\core{\int_0^{s_n} f\pare{Z_n(t)}dt},
$$
is given by 
$$
 \lim_{n\to \infty}  u_n(x)= \lim_{n\to \infty}  -\mathbb E^x\core{\int_0^{s_n} f\pare{Z_n(t)}dt} =
 u_A(x) + u_B(x), \qquad x\in\Omega.
$$

\begin{remark} \label{rem.K}{\rm 
First, let us see that the particle
has a uniformly positive probability of being killed in less than a finite number of jumps. 
To this end recall that we assumed hypothesis ({\bf H}) and then we have that there exists $c>0$ and $\delta>0$ such that
$$
V(x,y)\geq c >0, \qquad \forall y \in B_\delta (x)
$$
for any of the three kernels, $V=J,G$ or $R$, and for every $x\in \overline{\Omega}$.

Now, we observe that for $x$ in a narrow strip around $\partial \Omega$, 
$$
x\in \Gamma_1 = \{ x\in \overline\Omega \, : \, \mbox{dist} (x, \partial\overline \Omega) <\delta/2 \},
$$
we have that the probability of jumping outside $\overline\Omega$ in one jump is strictly positive (uniformly in the strip). 
With a similar argument, we get that the probability of reaching $\Gamma_1$ in one jump starting from a point in
$\Gamma_2 = \{ x\in \overline\Omega \, : \, \mbox{dist} (x, \Gamma_1) <\delta/2 \}$ is also strictly positive (uniformly)
and therefore the probability of exiting the domain in two jumps starting from $x\in \Gamma_2$ is uniformly positive
(we have a positive probability of reaching $\Gamma_1$ in the first jump and a positive probability of exiting from there
in the next jump). Repeating this argument we find a finite number of strips 
$\Gamma_j =\{ x\in \overline\Omega \, : \, \mbox{dist} (x, \Gamma_{j-1}) <\delta/2 \}$, 
$j=1,...,K$, such that $\overline\Omega = \cup_{j=1}^K \Gamma_j$ (recall that we assumed that $\Omega$ is bounded), 
with a uniform positive probability of reaching $\Gamma_{j-1}$ starting at a point in $\Gamma_j$ in one jump.
Hence, we have that the probability of reaching $\mathbb{R}^N \setminus \overline\Omega$ in less than $K+1$ jumps
starting at any point in $\overline{\Omega}$
is uniformly bounded below.

The property that the particle jumps outside $\overline\Omega$ with a uniform positive probability in a finite number of jumps
can be also obtained when, for example, $J\equiv 0$ (there are no jumps from $A_n$ to $A_n$), 
but assuming that $R$ and $G$ have a large support. 
For example, we have that 
$$
\mathbb{P} \Big(\mbox{jumping from $\overline\Omega$ to $\mathbb{R}^N\setminus \overline\Omega$ in one jump}\Big)\geq 
\min \left\{ \int_{B_n \cap\overline \Omega} \int_{
A_n \cap\overline \Omega^c} R(x,y) \, dydx;  \int_{A_n \cap \overline\Omega} \int_{
B_n \cap \overline\Omega^c} R(x,y) \, dydx \right\} >0,
$$
is strictly positive in $\overline{\Omega}$ when the support of $R$ is large. Similar bounds can be obtained for 
the exiting probability in two jumps, passing first from $A_n\cap \overline\Omega$ to $B_n\cap\overline \Omega$ using $R$ 
and then from $B_n \cap \overline\Omega$ to $\mathbb{R}^N\setminus\overline \Omega$, using $R$ or $G$, etc.

Now, when we have that the probability of exiting the domain in $K+1$ jumps 
is strictly bounded below by a positive constant we can use our previous reasoning to obtain 
that the stopping times are finite almost surely. 
Let us consider the process $W_n(t)$ bound from $Z_n(t)$ in the following way: 
for $W_n(t)$ we stay at the location until the previous process $Z_n(t)$ jumps $K+1$ times and at this time
we move the position of the particle to the position of $Z_n (t_{K+1})$ or we kill the particle if according
to the process $Z_n(t)$ we humped outside of $\overline\Omega$ in the first $K+1$ jumps
(notice that $W_n$ makes only one jump while $Z_n$ makes $K+1$ jumps in the same time interval). 
Call $\widetilde{s}_n$ the stopping time
associated with $W_n(t)$, that is, the first time at which $W_n(t)$ jumps outside $\overline\Omega$.
We trivially have 
$$
\widetilde{s}_n \geq s_n
$$
(since when the process $Z_n(t)$ jumps outside $\overline\Omega$ the process $W_n(t)$ also jumps outside 
but at a later time (it has to wait $K+1$ rings of the clock to move)). 

This process $W_n(t)$ is also a jump process in $\mathbb{R}^N$ (with a generator $\widetilde{L}_n$
that can be obtained iterating $L_n$ $K+1$ times) that has a uniform positive probability
of exiting the domain $\overline\Omega$ in only one jump (notice that we have proved that $Z_n(t)$ has a positive
uniform probability of exiting in less than $K+1$ jumps). Therefore, our previous arguments can be 
applied to this setting (we have a jump process with a uniformly positive probability of exiting in one jump)
and we obtain that
$$
\mathbb E^x\core{{s}_n} \leq \mathbb E^x\core{\widetilde{s}_n}<\infty,
$$
proving that 
$$
P^x\pare{s_n <\infty}=1,\qquad x\in \overline\Omega.
$$
From this point, the proof of the desired statement,
$$
u_n(x)=-\mathbb E^x\core{\int_0^{s_n} f\pare{Z_n(t)}dt}, \qquad x\in\overline\Omega,
$$
runs exactly as before.
}
\end{remark}

	\section{The Neumann problem}
	\label{Sect.proofs}
	\setcounter{equation}{0}
	
	In this section we present the proofs of the results concerning to the Neumann problem \eqref{1.1}.
	First, we need to establish an auxiliary result that helps us to prove existence and uniqueness of solutions
	to  \eqref{1.1} as well as to obtain a uniform bound (uniform in terms of being independent of $n$) for the solutions.

	Now, let us consider the generalized eigenvalue for the Neumann problem \eqref{1.1}. Let us introduce 
	$$
	\Phi(A,B,V)w := \int_{A}\int_{B}V(x,y)(w(y)-w(x))^2 dy dx
	$$
	which is set for any open bounded sets $A$ and $B \subset \R^N$, $V \in \mathcal{C}(\R^N,\R)$ and $w \in L^2(\R^N)$.
	
	We consider the following quantity 
	\begin{equation}\label{lambdaN}
	\lambda_n(\Omega)=\frac{1}{2}\inf_{u \in W, \; u \neq 0} 
	\frac{\Phi(A_n,A_n,J)u + 2\Phi(A_n,B_n,R)u + \Phi(B_n,B_n,G)u}{||u||^2_{L^2(\Omega)}}
	\end{equation}
	where 
	$W = \left\{ u \in L^2(\Omega) \, : \, \int_\Omega u(x) dx = 0 \right\}$
	is the space of the measurable functions in $L^2(\Omega)$ with null average set in \eqref{SetW_solutions}. We have the following:
	
	\begin{lemma}\label{lemmaeigenvalue}
		Let $\{\lambda_n(\Omega) \}$ be the family of values given by \eqref{lambdaN}. 
		Then, there exists a positive constant $c$ such that
		$$
		\lambda_n(\Omega) > c,\quad\forall n\geq 1.
		$$
	\end{lemma}
	\begin{proof}
		Let
		$$
		\min\{ J(z) , G(z), R(z)\} := K(z), 
		$$
		and assume that 
		\begin{equation*} \label{cotaK}
		K(z) \geq c_0 >0 \qquad \mbox{for } |z| \leq \delta
		\end{equation*}
		for some positive constant $c_0$. Notice that it is possible due to assumption ({\bf H}).

		Let us show that there exists a constant $c>0$ (independent of $n$) such that
		$$\lambda_n(\Omega)\geq c>0.$$
		Indeed, using that $\int_\Omega u_n(x)\, dx=0$, it follows from \cite[Proposition 3.4]{ElLibro} that 
		\begin{equation*}
		\begin{array}{l}
		\displaystyle \frac12 \int_{A_n} \int_{A_n} J (x,y) (u_n (y) - u_n (x))^2 dy dx  \displaystyle 
		+ \int_{A_n} \int_{B_n} R (x,y) (u_n (y) - u_n (x))^2 dy dx \\[10pt]
		\displaystyle + \frac12 \int_{B_n} \int_{B_n} G (x,y) (u_n (y) - u_n (x))^2 dy dx\\[10pt]
		\displaystyle \qquad \geq  \displaystyle \frac12 \int_{A_n} \int_{A_n} K (x,y) (u_n (y) - u_n (x))^2 dy dx  \displaystyle 
		+ \int_{A_n} \int_{B_n} K (x,y) (u_n (y) - u_n (x))^2 dy dx \\[10pt]
		\displaystyle \qquad \qquad  + \frac12 \int_{B_n} \int_{B_n} K (x,y) (u_n (y) - u_n (x))^2 dy dx\\[10pt]
		\qquad \displaystyle \geq c_0 \int_\Omega \int_\Omega \chi_{|x-y|<\delta} (u_n (y) - u_n (x))^2 dy dx\\[10pt]
		\qquad  \displaystyle \geq c \| u_n \|_{L^2(\Omega)}^2
		\end{array}
		\end{equation*}
		for some constant $c>0$, independent of $n$, 
		where $\chi_{|x-y|<\delta}$ is the characteristic function of the set $\mathcal{O}(x) = \{ y \in \R^N \; : \; |x-y|<\delta \}$.
	\end{proof}
	
	\begin{remark} {\rm
		Notice that we need $R\neq 0$ in order to get $\lambda_n(\Omega) > 0$. In fact, if $R=0$, then 
		$$
		u_n(x) = \left\{
		\begin{array}{ll}
		\displaystyle \frac{n}{|A_n|} \qquad & x \in A_n \\[10pt]
		\displaystyle - \frac{n}{|B_n|} \qquad & x \in B_n
		\end{array}
		\right.
		$$
		is a function such that
		$$
		\int_\Omega u_n (x) \, dx=0,
		$$
		as $R=0$ and $u_n$ is constant in $A_n$ and in $B_n$ we also have
		\begin{equation*}
		\begin{array}{l}
		\displaystyle  \displaystyle \frac12 \int_{A_n} \int_{A_n} J (x,y) (u_n (y) - u_n (x))^2 dy dx  
		+ \int_{A_n} \int_{B_n} R (x,y) (u_n (y) - u_n (x))^2 dy dx \\[10pt]
		\displaystyle \qquad + \frac12 \int_{B_n} \int_{B_n} G (x,y) (u_n (y) - u_n (x))^2 dy dx 
		=0, 
		\end{array}
		\end{equation*}
		but  
		$$
		\| u_n \|_{L^2(\Omega)} \neq 0.
		$$
		
		However, $J$ or $G$ could be zero as long as the support of $R$ is large enough. 
		For example, assume that $J\equiv 0$ and that the support of $R$ is such that 
		for every $x\in A_n$ there exists $y\in B_n$ such that $R(x,y)\geq c >0$ with $c$ independent of $n$.
		This condition holds for example when $R$ is strictly positive in $\Omega \times \Omega$.
		
		Now, given a function $u_n$ we can consider $w_n = u_n -c$ with $c\in \mathbb{R}$ a constant
		such that $\int_{B_n} w_n =0$. Take a point $x\in A_n$, a point $y\in B_n$ such that $R(x,y)>0$ and a small 
		radius $r$ such that $R(x',y')\geq c/2 >0$ for every $x'\in B_r(x)\cap A_n$ and every $y'\in B_r(y)\cap B_n$. Then we have
		$$
		\int_{B_r(x)\cap A_n} |w_n(x')|^2 dx' \leq C \int_{B_r(x)\cap A_n} \int_{B_r(y)\cap B_n} R(x',y') |w_n(y')-w_n({x'})|^2 dy'dx'
		+ C \int_{B_r(x)\cap B_n} |w_n(y')|^2 dy'.
		$$
		Using a covering argument we get 
		$$
		\int_{A_n} |w_n(x')|^2 dx' \leq C \int_{ A_n} \int_{B_n} R(x',y') |w_n(y')-w_n({x'})|^2 dy'dx'
		+ C \int_{B_n} |w_n(y')|^2 dy'.
		$$
		Now, since $\int_{B_n} w_n =0$, from \cite[Proposition 3.4]{ElLibro} we have that 
		$$
		 \int_{B_n} |w_n(y')|^2 dy' \leq C \int_{B_n}\int_{B_n} G(x',y') |w_n(y')-w_n({x'})|^2 dy'dx'.
		$$
		Hence we arrive to
		$$
	\begin{array}{l}	
		\displaystyle \int_\Omega |w_n|^2 = \int_{A_n} |w_n(x')|^2 dx' + \int_{B_n} |w_n(y')|^2 dy' \\[10pt]
		\qquad \displaystyle \leq C \int_{ A_n} \int_{B_n} R(x',y') |w_n(x')-w_n(y')|^2 dy'dx' +
		C \int_{B_n}\int_{B_n} G(x',y') |w_n(x')-w_n(y')|^2 dy'dx'.
		\end{array}
		$$
		In terms of $u_n$, using that $\int_{\Omega} u_n =0$, we have that 
		$$
	\begin{array}{l}	
		\displaystyle \int_\Omega |u_n|^2 = \min_{c\in \mathbb{R}} \int_\Omega |u_n-c|^2
		\leq \int_\Omega |w_n|^2
		 \\[10pt]
		\qquad \displaystyle \leq C \int_{ A_n} \int_{B_n} R(x',y') |u_n(x')-u_n(y')|^2 dy'dx' +
		C \int_{B_n}\int_{B_n} G(x',y') |u_n(x')-u_n(y')|^2 dy'dx'.
		\end{array}
		$$
		This inequality implies that $\lambda_n(\Omega) > 0$. 
		
		We prefer to state our results under the assumption 
		 ({\bf H}) for simplicity.

		On the other hand, if we have $J\equiv 0$ but the support of $R$ is small in the sense that there is an 
		open subset $D_n\subset A_n$ such that $R(x,y)=0$ for every $x\in D_n$, and every $y\in B_n$ then 
		it holds that $\lambda_n(\Omega) = 0$. To prove this, just consider $u_n=\chi_{D_n}$, the characteristic function of $D_n$, and observe that
		$$
		 \int_\Omega |u_n|^2 = |D_n| \neq 0,
		$$
		but 
		$$
		\int_{ A_n} \int_{B_n} R(x',y') |u_n(x')-u_n(y')|^2 dy'dx' +
		\int_{B_n}\int_{B_n} G(x',y') |u_n(x')-u_n(y')|^2 dy'dx' =0.
		$$
		}
	\end{remark}

	As a consequence of Lemma \ref{lemmaeigenvalue}, we obtain existence and uniqueness of the solutions of the Neumann problem \eqref{1.1}.
	\begin{theorem} \label{EUN}
		Let $W \subset L^2(\Omega)$ be the closed subspace given by \eqref{SetW_solutions} and assume conditions $({\bf H})$ under the non-singular kernels $J$, $R$ and $G$. 
		Then, for each $f \in W$, there exists a unique $u \in W$ satisfying equation \eqref{1.1} and being the minimizer of the functional \eqref{1.1.energy}.
	\end{theorem}
	\begin{proof}
		Let $a: W \times W \mapsto \R$ be the following bilinear form
		\begin{equation}  \label{bforma}
		\begin{array}{rl}
		a(u,v) = & \displaystyle \int_{A_n} v(x) \int_{A_n} J (x,y) (u (y) - u (x)) dy dx 
		%\\[10pt]
		%& \displaystyle 
		+ \int_{A_n} v(x) \int_{B_n} R (x,y) (u (y) - u (x)) dy dx
		\\[10pt]
		& \displaystyle 
		+ \int_{B_n} v(x) \int_{A_n} R (x,y) (u (y) - u (x)) dy dx
		%\\[10pt]
		%& \displaystyle 
		+ \int_{B_n} v(x) \int_{B_n} G (x,y) (u (y) - u (x)) dy dx.
		\end{array}
		\end{equation}
		
		It is not difficult to see that $a$ is continuous, symmetric and coercive by Lemma \ref{lemmaeigenvalue}. 
		Thus, it follows from Lax-Milgram Theorem that there is a unique $ u \in W$ satisfying 
		$$
		a(u,v) = \int_\Omega f(x) \, v(x) \, dx, \quad \forall v \in W
		$$
		and each $f \in W$ given. Also, we have that the function $u$ is the minimizer of the energy \eqref{1.1.energy}.
		
		Finally, we conclude that $u$ is the unique solution of \eqref{1.1} since $L^2(\Omega) = W \oplus [1]$ and 
		$$
		\begin{gathered}
		\int_{A_n} \int_{A_n} J(x,y)(u(y) - u(x)) dy dx + \int_{A_n} \int_{B_n} R(x,y)(u(y) - u(x)) dy dx 
		+ \int_{B_n} \int_{A_n} R(x,y)(u(y) - u(x)) dy dx \\
		+ \int_{B_n} \int_{B_n} J(x,y)(u(y) - u(x)) dy dx - \int_\Omega f(x) dx = 0.
		\end{gathered}
		$$
	\end{proof}

	\subsection{Proof of Theorem \ref{theo1.intro}}
	
	Now, we are ready to prove our homogenization result.
	
	\begin{proof}[Proof of Theorem \ref{theo1.intro}]
		{\bf Uniform bounds.} 
		
		It follows from \eqref{1.1} that 
		\begin{equation} \label{variationalproblem}
		\begin{array}{rl}
		& \displaystyle \int_\Omega \varphi f(x)dx = \displaystyle \int_{A_n} \varphi(x) \int_{A_n} J (x,y) (u_n (y) - u_n (x)) dy dx 
		%\\[10pt]
		%& \displaystyle 
		+ \int_{A_n} \varphi(x) \int_{B_n} R (x,y) (u_n (y) - u_n (x)) dy dx
		\\[10pt]
		& \displaystyle 
		\qquad + \int_{B_n} \varphi(x) \int_{A_n} R (x,y) (u_n (y) - u_n (x)) dy dx
		%\\[10pt]
		%& \displaystyle 
		+ \int_{B_n} \varphi(x) \int_{B_n} G (x,y) (u_n (y) - u_n (x)) dy dx
		\end{array}
		\end{equation}
		for all $\varphi\in L^2(\Omega)$ and $n\geq1$.  
		
		Thus, if we take $\varphi = u_n$ in \eqref{variationalproblem}, we get from Lemma \ref{lemmaeigenvalue} that there exists $c>0$ independent of $n$ such that 
		\begin{equation*}
		\begin{array}{rl}
		\displaystyle \| u_n \|_{L^2(\Omega)} \| f \|_{L^2(\Omega)} 
		\geq & \displaystyle \frac12 \int_{A_n} \int_{A_n} J (x,y) (u_n (y) - u_n (x))^2 dy dx \\[10pt]
		& \displaystyle + \int_{A_n} \int_{B_n} R (x,y) (u_n (y) - u_n (x))^2 dy dx \\[10pt]
		& \displaystyle + \frac12 \int_{B_n} \int_{B_n} G (x,y) (u_n (y) - u_n (x))^2 dy dx\\[10pt]
		\displaystyle \geq & c \| u_n \|_{L^2(\Omega)}^2,
		\end{array}
		\end{equation*}
		and then, 
		\begin{equation} \label{eq270}
		\| u_n \|_{L^2(\Omega)} \leq \frac1c \| f \|_{L^2(\Omega)} 
		\end{equation}
		which means that $\| u_n \|_{L^2(\Omega)}$ is uniformly bounded by a positive constant independent of $n$.

		{\bf Limit of equation \eqref{variationalproblem} as $n \to \infty$.}
		
		Since the solutions $u_n$ are uniformly bounded, it follows from \eqref{eq270} that $\chi_{A_n}u_n$ and $\chi_{B_n}u_n$ are also uniformly bounded. Thus, there are $u_A$, $u_B\in L^2(\Omega)$ such that, up to subsequences,
		\begin{equation}\label{convsolution}
		\begin{gathered}
		\chi_{A_n}u_n\rightharpoonup u_A 
		\quad \textrm{ and } \quad
		\chi_{B_n}u_n\rightharpoonup u_B 
		\end{gathered}
		\end{equation}
		weakly in $L^2(\Omega)$.
		Also, observe that
		\begin{equation*}
		\begin{gathered}
		\int_\Omega\chi_{A_n}(y) \, V(x,y) dy\to \int_\Omega X(y) \, V(x,y) dy\\
		\mbox{and} \\
		\int_\Omega\chi_{B_n}(y) \, V(x,y) dy\to \int_\Omega (1-X)(y) \, V(x,y) dy
		\end{gathered}
		\end{equation*}
		for all $x\in\mathbb{R}^n$ where $V$ can be any one of the kernels $J$, $R$ or $G$, and then, we have by the Dominated Convergence Theorem that 
		\begin{equation}\label{convkernel}
		\begin{gathered}
		\int_\Omega\chi_{A_n}(y) \, V(\cdot -y) dy\to \int_\Omega X(y) \, V(\cdot-y) dy \\
		\mbox{and} \\
		\int_\Omega\chi_{B_n}(y) \, V(\cdot-y) dy\to \int_\Omega (1-X)(y) \, V(\cdot-y) dy
		\end{gathered}
		\end{equation}
		strongly in $L^2(\Omega)$ as $n \to \infty$.
		
		To pass to the limit equation in \eqref{variationalproblem}, let us rewrite it as
		\begin{equation} \label{variationalproblem01}
		\begin{array}{rl}
		\displaystyle \int_\Omega \varphi f(x)dx  = & \displaystyle \int_\Omega\chi_{A_n} \varphi(x) \int_\Omega\chi_{A_n} J (x,y) u_n (y) dydx
		 %\\[10pt]
		 %\displaystyle & 
		 -\displaystyle \int_\Omega\chi_{A_n}\varphi(x)u_n(x)\int_\Omega\chi_{A_n}J(x,y) dy dx 
		 \\[10pt]
		& \displaystyle 
		+ \int_\Omega\chi_{A_n} \varphi(x) \int_\Omega\chi_{B_n} R (x,y) u_n (y)  dy dx 
		%\\[10pt]
		%\displaystyle &
		-\displaystyle \int_\Omega\chi_{A_n} \varphi(x) u_n(x)\int_\Omega\chi_{B_n} R (x,y) dydx
		\\[10pt]
		& \displaystyle 
		+ \int_\Omega\chi_{B_n} \varphi(x) \int_\Omega\chi_{A_n} R (x,y) u_n (y)  dy dx 
		%\\[10pt]\displaystyle &
		-\displaystyle \int_\Omega\chi_{B_n} \varphi(x) u_n(x) \int_\Omega\chi_{A_n} R (x,y) dy dx
		\\[10pt]
		& \displaystyle 
		+ \int_\Omega\chi_{B_n} \varphi(x) \int_\Omega\chi_{B_n} G (x,y) u_n (y)  dy dx 
		%\\[10pt]\displaystyle &
		-\displaystyle \int_\Omega\chi_{B_n} \varphi(x) u_n(x) \int_\Omega\chi_{B_n} G (x,y) dy dx.
		\end{array}
		\end{equation}
		
		Taking in account \eqref{convsolution}, \eqref{convkernel} in equation \eqref{variationalproblem01}, as $n\to\infty$, leads us to
		\begin{equation} \label{prelimit}
		\begin{array}{rl}
		\displaystyle \int_\Omega \varphi f(x)dx = & \displaystyle \int_\Omega X(x) \varphi(x) \int_\Omega J (x,y) u_A (y) dydx
		%\\[10pt]\displaystyle & 
		-\displaystyle\int_\Omega\varphi(x)u_A(x)\int_\Omega X(y)J(x,y) dy dx 
		\\[10pt]
		& \displaystyle 
		+ \int_\Omega X(x) \varphi(x) \int_\Omega R (x,y) u_B(y)  dy dx
		%\\[10pt]\displaystyle &
		-\displaystyle\int_\Omega \varphi(x) u_A(x)\int_\Omega(1-X(y)) R (x,y) dydx
		\\[10pt]
		& \displaystyle 
		+ \int_\Omega (1-X(x)) \varphi(x) \int_\Omega R (x,y) u_A (y)  dy dx
		%\\[10pt]\displaystyle &
		-\displaystyle\int_\Omega  \varphi(x) u_B(x) \int_\Omega X(y) R (x,y) dy dx
		\\[10pt]
		& \displaystyle 
		+ \int_\Omega (1-X(x)) \varphi(x) \int_\Omega G (x,y) u_B (y)  dy dx
		%\\[10pt]\displaystyle &
		-\displaystyle\int_\Omega\varphi(x) u_B(x) \int_\Omega (1-X(y)) G (x,y) dy dx
		\end{array}
		\end{equation}
		
		Now, we rewrite each integral of the right hand side of \eqref{prelimit} as
		\begin{equation}\label{prelimit1}
		\begin{gathered}
		\int_\Omega X(x) \varphi(x) \int_\Omega J (x,y) u_A (y) dydx  -\int_\Omega\varphi(x)u_A(x)\int_\Omega X(y)J(x,y) dy dx \\
		=\int_\Omega\varphi(x)\int_{\Omega}J(x,y)\left(X(x)u_A(y)-X(y)u_A(x) \right)dydx,
		\end{gathered}
		\end{equation}
		\begin{equation}\label{prelimit2}
		\begin{gathered}
		\int_\Omega (1-X(x)) \varphi(x) \int_\Omega G (x,y) u_B (y)  dy dx - \int_\Omega\varphi(x) u_B(x) \int_\Omega (1-X(y)) G (x,y) dy dx\\
		= \int_\Omega\varphi(x)\int_\Omega G(x,y)\left[ (1-X(x))u_B(y)- (1-X(y))u_B(x) \right] dydx,
		\end{gathered}
		\end{equation}
		and 
		\begin{equation}\label{prelimit3}
		\begin{array}{l}
		\displaystyle 
		\int_\Omega X(x) \varphi(x) \int_\Omega R (x,y) u_B(y)  dy dx-\displaystyle\int_\Omega \varphi(x) u_A(x)\int_\Omega(1-X(y)) R (x,y) dydx 
		\\[10pt]
		\displaystyle \qquad 
		+ \int_\Omega (1-X(x)) \varphi(x) \int_\Omega R (x,y) u_A (y)  dy dx-\displaystyle\int_\Omega  \varphi(x) u_B(x) \int_\Omega X(y) R (x,y) dy dx
		\\[10pt]
		\displaystyle 
		=\int_\Omega\varphi(x)\int_\Omega R(x,y)\left[X(x)u_B(y)-(1-X(y))u_A(x)+(1-X(x))u_A(y)-X(y)u_B(x)\right] dydx
		\\[10pt]
		\displaystyle 
		=\int_\Omega\varphi(x)\int_\Omega R(x,y)\left[(1-X(x))u_A(y)-(1-X(y))u_A(x)+X(x)u_B(y)-X(y)u_B(x)\right] dydx.
		\end{array}
		\end{equation}
		Putting together \eqref{prelimit}, \eqref{prelimit1}, \eqref{prelimit2} and \eqref{prelimit3}, we obtain
		\begin{equation} \label{eq660}
		\begin{array}{rl}
		\displaystyle\int_\Omega f(x)\varphi(x) dx&=\displaystyle\int_\Omega\varphi(x)\int_{\Omega}J(x,y)\left(X(x)u_A(y)-X(y)u_A(x) \right)dydx\\[10pt]
		&\quad +\displaystyle\int_\Omega\varphi(x)\int_\Omega R(x,y)\left[(1-X(x))u_A(y)-(1-X(y))u_A(x)
		+X(x)u_B(y)-X(y)u_B(x)\right] dydx\\[10pt]
		&\quad +\displaystyle\int_\Omega\varphi(x)\int_\Omega G(x,y)\left[ (1-X(x))u_B(y)- (1-X(y))u_B(x) \right] dydx
		\end{array}
		\end{equation}
		which gives us the homogenized equation \eqref{limf}.
		
		{\bf Limit for test functions $\chi_{A_n}\varphi$.} \label{proofpart3} 
		
		Now let us consider \eqref{variationalproblem01} taking test functions as $\chi_{A_n}\varphi$. Then, we have
		\begin{equation*} 
		\begin{array}{rl}
		\displaystyle \int_\Omega\chi_{A_n} \varphi f(x)dx = & \displaystyle \int_\Omega\chi_{A_n} \varphi(x) \int_\Omega\chi_{A_n} J (x,y) u_n (y) dydx 
		%\\[10pt]\displaystyle & 
		-\displaystyle \int_\Omega\chi_{A_n}\varphi(x)u_n(x)\int_\Omega\chi_{A_n}J(x,y) dy dx 
		\\[10pt]
		& \displaystyle 
		+ \int_\Omega\chi_{A_n} \varphi(x) \int_\Omega\chi_{B_n} R (x,y) u_n (y)  dy dx 
		%\\[10pt]\displaystyle &
		-\displaystyle \int_\Omega\chi_{A_n} \varphi(x) u_n(x)\int_\Omega\chi_{B_n} R (x,y) dydx
		\end{array}
		\end{equation*}
		for any $\varphi \in L^2(\Omega)$.
		
		Passing to the limit the above equation leads us to
		\begin{equation*}
		\begin{array}{rl}
		\displaystyle \int_\Omega X \varphi f(x)dx = & \displaystyle \int_\Omega X \varphi(x) \int_\Omega J (x,y) u_A (y) dydx 
		%\\[10pt]\displaystyle & 
		-\displaystyle \int_\Omega\varphi(x)u_A(x)\int_\Omega X(y)J(x,y) dy dx 
		\\[10pt]
		& \displaystyle 
		+ \int_\Omega X \varphi(x) \int_\Omega R (x,y) u_B (y)  dy dx 
		%\\[10pt]\displaystyle &
		-\displaystyle \int_\Omega \varphi(x) u_A(x)\int_\Omega(1-X(y)) R (x,y) dydx.
		\end{array}
		\end{equation*}
		Hence, due to \eqref{prelimit1} and \eqref{prelimit3}, we get that 
		\begin{equation*}\label{eqlimitA}
		\begin{array}{rl}
		\displaystyle \int_\Omega X \varphi f(x)dx = & \displaystyle\int_\Omega\varphi(x)\int_{\Omega}J(x,y)\left(X(x)u_A(y)-X(y)u_A(x) \right)dydx \\[10pt]
		& \displaystyle 
		+ \int_\Omega  \varphi(x) \int_\Omega R (x,y) \left[X(x)u_B (y)-u_A(x)(1-X(y)) \right] dy dx,
		\end{array}
		\end{equation*}
		which proves limit equation \eqref{limitAn}.
		
		{\bf Limit for test functions $\chi_{B_n}\varphi$.}
		
		Finally we consider \eqref{variationalproblem01} taking test functions $\chi_{B_n}\varphi$. We have 
		\begin{equation*}
		\begin{array}{rl}
		\displaystyle \int_\Omega \chi_{B_n}\varphi f(x)dx = 
		& \displaystyle 
		\int_\Omega\chi_{B_n} \varphi(x) \int_\Omega\chi_{A_n} R (x,y) u_n (y)  dy dx 
		%\\[10pt]\displaystyle &
		-\displaystyle \int_\Omega\chi_{B_n} \varphi(x) u_n(x) \int_\Omega\chi_{A_n} R (x,y) dy dx
		\\[10pt]
		& \displaystyle 
		+ \int_\Omega\chi_{B_n} \varphi(x) \int_\Omega\chi_{B_n} G (x,y) u_n (y)  dy dx 
		%\\[10pt]\displaystyle &
		-\displaystyle \int_\Omega\chi_{B_n} \varphi(x) u_n(x) \int_\Omega\chi_{B_n} G (x,y) dy dx.
		\end{array}
		\end{equation*}
		Hence, we can argue as in \ref{proofpart3} to get 
		\begin{equation*}  
		\begin{array}{rl}
		\displaystyle \int_\Omega (1-X(x))\varphi f(x)dx & = 
		\displaystyle 
		\int_\Omega\varphi(x) \int_\Omega  R (x,y) \left[u_A (y)(1-X(x)) -u_B(x)X(y) \right]dy dx 
		\\[10pt]
		& \qquad \displaystyle 
		+ \int_\Omega\varphi(x) \int_\Omega G (x,y)\left[ u_B (y)(1-X(x))-u_B(x)(1-X(y))\right]  dy dx
		\end{array}
		\end{equation*}
		for all $\varphi \in L^2(\Omega)$.

		Finally, we prove that the solutions to the system \eqref{limitAn}--\eqref{limitBn} are unique. 
		For this purpose, let $(u_A,u_B),(v_A, v_B)\in L^2(\Omega) \times L^2(\Omega)$ with $\int_{\Omega} \{ u_A+u_B \} dx = \int_{\Omega} \{ v_A+v_B \} dx=0$ be two solutions of this system. 
		
		Set $w_A=u_A-v_A$ and $w_B=u_B-v_B$. Then, from \eqref{limitAn} and \eqref{limitBn} one has
		\begin{equation}\label{uniquewA}
		0=\displaystyle\int_{\Omega} J(x,y)[X(x)w_A(y)-X(y)w_A(x)] dy +\displaystyle\int_{\Omega} R(x,y)[X(x)w_B(y)-(1-X(y))w_A(x)] dy
		\end{equation}
		and
		\begin{equation}\label{uniquewB}
		0=\displaystyle\int_{\Omega} R(x,y)[(1-X(x))w_A(y)-X(y)w_B(x)] dy
		+\displaystyle\int_{\Omega} G(x,y)[(1-X(x))w_B(y)-(1-X(y))w_B(x)] dy.
		\end{equation}
		
		Without loss of generality, we can suppose $0<X(x)<1$ a.e. $x\in\Omega$, since from equations \eqref{uniquewA} and \eqref{uniquewB}, we obtain $w_A=w_B=0$ respectively in the sets $\{x\in\Omega:X(x)=0\}$ and $\{x\in\Omega:X(x)=1\}$. 
		
		Now, multiplying \eqref{uniquewA} and \eqref{uniquewB} by ${w_A}/{X}$ and ${w_B}/{(1-X)}$ respectively, and integrating in $\Omega$, we get  
		\begin{eqnarray*}
		&&\begin{array}{rl}
		0=&\displaystyle\int_{\Omega}\left(\dfrac{w_A}{X}\right)(x)\int_{\Omega} J(x,y)X(x)X(y)\left[\left(\dfrac{w_A}{X}\right)(y)-\left(\dfrac{w_A}{X}\right)(x)\right] dy dx\\[10pt]
		&+\displaystyle\int_{\Omega}\left(\dfrac{w_A}{X}\right)(x)\int_{\Omega} R(x,y)[X(x)w_B(y)-(1-X(y))w_A(x)] dy dx 
		\end{array} \\
		\quad \textrm{ and } \\
		&&\begin{array}{rl}
		0=&\displaystyle\int_{\Omega}\left(\dfrac{w_B}{1-X}\right)(x)\int_{\Omega} R(x,y)[(1-X(x))w_A(y)-X(y)w_B(x)] dy\\[10pt]
		&+\displaystyle\int_{\Omega}\left(\dfrac{w_B}{1-X}\right)(x)\int_{\Omega} G(x,y)(1-X(x))(1-X(y))\left[\left(\dfrac{w_B}{1-X}\right)(y)-\left(\dfrac{w_B}{1-X}\right)(x)\right] dy.
		\end{array}
		\end{eqnarray*}
		Now, we can rewrite the above equations as
		\begin{eqnarray*}
		&&\begin{array}{rl}
		0=&-\dfrac{1}{2}\displaystyle\int_{\Omega\times\Omega} J(x,y)X(x)X(y)\left[\left(\dfrac{w_A}{X}\right)(y)-\left(\dfrac{w_A}{X}\right)(x)\right]^2 dy dx\\[10pt]
		&+\displaystyle\int_{\Omega\times\Omega} R(x,y)\left[w_A(x)w_B(y)-\dfrac{1-X(y)}{X(x)}w_A^2(x)\right] dy dx
		\end{array} \\ \quad \textrm{ and } \\
		&&\begin{array}{rl}
		0=&\displaystyle\int_{\Omega\times\Omega} R(x,y)\left[w_A(x)w_B(y)-\dfrac{X(x)}{1-X(y)}w_B^2(y)\right] dy\\[10pt]
		&-\dfrac{1}{2}\displaystyle\int_{\Omega\times\Omega} G(x,y)(1-X(x))(1-X(y))\left[\left(\dfrac{w_B}{1-X}\right)(y)-\left(\dfrac{w_B}{1-X}\right)(x)\right]^2 dy.
		\end{array}
		\end{eqnarray*}
		
		Finally, we sum the above equalities to obtain
		\begin{eqnarray*}
		0 & = & 
		\dfrac{1}{2}\displaystyle\int_{\Omega\times\Omega} J(x,y)X(x)X(y)\left[\left(\dfrac{w_A}{X}\right)(y)-\left(\dfrac{w_A}{X}\right)(x)\right]^2 dy dx\\
		& &+ \int_{\Omega\times\Omega} R(x,y)(1-X(y))X(x)\left[\left(\frac{w_B}{1-X}\right)(y)-\left(\frac{w_A}{X(x)}\right)(x)\right]^2 dy dx\\
		& &+ \dfrac{1}{2}\displaystyle\int_{\Omega\times\Omega} G(x,y)(1-X(x))(1-X(y))\left[\left(\dfrac{w_B}{1-X}\right)(y)-\left(\dfrac{w_B}{1-X}\right)(x)\right]^2 dy.
		\end{eqnarray*}
		
		Hence, we can conclude that there exists a constant $c$ such that 
		$$
		\dfrac{w_A}{X}(x)=\dfrac{w_B}{1-X}(x)= c, \qquad \forall x \in \Omega.
		$$
		Since 
		$$0=\int_\Omega \{ w_A+w_B \} dx = \int_\Omega c \{ X + (1-X) \} dx = c \, |\Omega|$$
		we obtain that $c=0$, and then, $w_A=w_B=0$ in $\Omega$ finishing the proof.
	\end{proof}

\subsection{Proof of Corollaries \ref{corX=0} and \ref{corX=1}}

Now we consider the extreme cases $X(x)=0$ or $X(x)=1$ proving Corollaries \ref{corX=0} and \ref{corX=1}.

		\begin{proof}[Proof of Corollary \ref{corX=0}]
		Suppose $X(x)=0$ a.e. in $x\in\R^n$. From the limit equation \eqref{limitAn}, we get that
		\begin{equation*} 
		u_A(x)\int_{\Omega} R (x,y) dy = 0, \quad \textrm{ a.e. } x \in \Omega.
		\end{equation*}
		Hence, as $\int_{\Omega} R (x,y) dy$ is a strictly positive function, we have that 
		\begin{equation} \label{eq675}
		u_A(x)=0\textrm{ a.e. } \forall x \in \Omega,
		\end{equation}
		which leads us to the limit equation
		\begin{equation} \label{eq682}
		\int_\Omega f(x)\varphi(x) dx=\int_\Omega\varphi(x)\int_\Omega G(x,y)\left( u_B(y)- u_B(x) \right) dydx
		\end{equation}
		where $u_B \in W$ is the unique solution of \eqref{eq682}. 
				
		Now, let us show that $u_n \to u_B$ strongly in $L^2(\Omega)$. Due to Lemma \ref{lemmaeigenvalue}, we have that there exists $c>0$, independent of $n$ such that the bilinear form $a$ introduced in \eqref{bforma} satisfies 
		\begin{equation} \label{eq685}
		\begin{gathered}
		c \| u_n-u_B \|_{L^2(\Omega)}^2 \leq a(u_n-u_B, u_n-u_B) 
		\leq a(u_n, u_n) - 2 a(u_n, u_B) + a(u_B, u_B).
		\end{gathered}
		\end{equation}
		
		Also, since $u_n$ satisfies \eqref{1.1}, we know that 
		\begin{equation} \label{eq695}
		\begin{gathered}
		a(u_n, u_n) = \int_\Omega f u_n \, dx \quad \to \quad \int_\Omega f ( u_A + u_B) \, dx = \int_\Omega f u_B \quad \textrm{ as } n \to \infty,\\ 
		\textrm{ and }  \quad 
		a\left( u_n, u_B \right) = \int_\Omega f u_B \, dx \quad \forall n. 
		\end{gathered}
		\end{equation}
		
		On the other side, we can pass to the limit in $a(u_B, u_B)$ as in \eqref{prelimit} obtaining 
		\begin{equation} \label{eq705}
		\begin{array}{l}
		\displaystyle 
		\lim_{n\to\infty} a(u_B, u_B) = \int_\Omega X(x) u_B(x) \int_\Omega J(x,y) X(y) ( u_B(y) - u_B(x) ) dy dx  
		\\[10pt]
		\displaystyle \qquad 
		+ \int_\Omega X(x) u_B(x) \int_\Omega R(x,y) (1-X(y)) ( u_B(y) - u_B(x) ) dy dx 
		\\[10pt]
		\displaystyle \qquad 
		+ \int_\Omega (1-X(x)) u_B(x) \int_\Omega R(x,y) X(y) ( u_B(y) - u_B(x) ) dy dx 
		\\[10pt]
		\displaystyle \qquad 
		+ \int_\Omega (1-X(x)) u_B(x) \int_\Omega G(x,y) (1-X(y)) ( u_B(y) - u_B(x) ) dy dx 
		\\[10pt]
		\displaystyle 
		= \int_\Omega u_B(x) \int_\Omega G(x,y) ( u_B(y) - u_B(x) ) dy dx  
		= \int_\Omega f u_B \, dx
		\end{array}
		\end{equation}
		since $X(x) = 0$ for all $x \in \Omega$ and $u_B$ satisfies \eqref{eq682}.
		
		Thus, it follows from \eqref{eq675}, \eqref{eq685}, \eqref{eq695} and \eqref{eq705} that 
		$$
		0 \leq c \lim_{n\to \infty} \| u_n-u_B \|_{L^2(\Omega)}^2 \leq \int_\Omega f u_B dx - 2 \int_\Omega f u_B dx + \int_\Omega f u_B dx = 0 
		$$
		proving the corollary. 
	\end{proof}

	Finally we observe that the proof of Corollary \ref{corX=1} is quite similar to the one of Corollary \ref{corX=0} and then is left to the reader.

\subsection{Proof of Theorem \ref{Corrector.Neumann}}

Here we give a proof for Theorem \ref{Corrector.Neumann} which sets a corrector family to the solutions of the Neumann problem \eqref{1.1}.

		\begin{proof}[Proof of Theorem \ref{Corrector.Neumann}]
		Let $a$ be the bilinear for given by \eqref{bforma}. First we notice that $\omega_n - \int_\Omega \omega_n \in W$. 
		Hence, as $a$ is a coercive form by Lemma \ref{lemmaeigenvalue}, we know that there exists $c>0$ such that 
		$$
		\begin{array}{ll}
		& \displaystyle c \left\|u_n - \left(\omega_n - \int_\Omega \omega_n \, dx \right) \right\|_{L^2(\Omega)}^2 \leq 
		a\left(u_n - \omega_n + \int_\Omega \omega_n, u_n - \omega_n + \int_\Omega \omega_n\right) \\[10pt]
		& \displaystyle \displaystyle \qquad \qquad \leq a(u_n - \omega_n, u_n - \omega_n)  
		+ 2 a\left(u_n - \omega_n, \int_\Omega \omega_n\right) + a\left(\int_\Omega \omega_n, \int_\Omega \omega_n\right) \\[10pt]
		& \qquad \qquad \leq a(u_n - \omega_n, u_n - \omega_n) 
		\end{array}
		$$
		since $\int_\Omega \omega_n$ is constant. 
		In this way, we will conclude the proof, if we show that 
		$$
		a(u_n - \omega_n, u_n - \omega_n) \to 0, \quad \textrm{ as } n \to \infty.
		$$
		
		Recall that  
		\begin{equation} \label{eq630}
		\begin{gathered}
		a(u_n - \omega_n, u_n - \omega_n) = a\left( u_n-\dfrac{\chi_{A_n}u_A}{X}-\dfrac{\chi_{B_n}u_B}{1-X}, u_n-\dfrac{\chi_{A_n}u_A}{X}-\dfrac{\chi_{B_n}u_B}{1-X} \right) \\
		= a(u_n, u_n) - 2 a\left( u_n, \dfrac{\chi_{A_n}u_A}{X} \right) 
		- 2 a\left( u_n, \dfrac{\chi_{B_n}u_B}{1-X} \right) + a\left( \dfrac{\chi_{A_n}u_A}{X}, \dfrac{\chi_{A_n}u_A}{X} \right) \\
		+ 2 a\left( \dfrac{\chi_{A_n}u_A}{X}, \dfrac{\chi_{B_n}u_B}{1-X} \right) + a\left( \dfrac{\chi_{B_n}u_B}{1-X}, \dfrac{\chi_{B_n}u_B}{1-X} \right).
		\end{gathered}
		\end{equation}
		
		We will pass to the limit in each term of the right hand side of the previous expression. The first three terms are easy to compute. As $n \to +\infty$, since $u_n$ satisfies \eqref{1.1}, we get 
		\begin{equation} \label{eq640}
		\begin{gathered}
		a(u_n, u_n) = \int_\Omega f u_n \, dx \quad \to \quad \int_\Omega f ( u_A + u_B) \, dx, \\ 
		a\left( u_n, \dfrac{\chi_{A_n}u_A}{X} \right) = \int_\Omega f \dfrac{\chi_{A_n}u_A}{X} \, dx \quad \to \quad \int_\Omega f u_A \, dx, \quad 
		\textrm{ and }  \\
		\qquad a\left( u_n, \dfrac{\chi_{B_n}u_B}{1-X} \right) = \int_\Omega f \dfrac{\chi_{B_n}u_B}{1-X} \, dx \quad \to \quad \int_\Omega f u_B \, dx.
		\end{gathered}
		\end{equation}
		
		Next, we observe that
		$$
		\begin{gathered}
		a\left( \dfrac{\chi_{A_n}u_A}{X}, \dfrac{\chi_{A_n}u_A}{X} \right) = 
		\int_{A_n} \dfrac{\chi_{A_n}u_A}{X} (x) \int_{A_n} J(x,y) \left(  \dfrac{\chi_{A_n}u_A}{X}(y) - \dfrac{\chi_{A_n}u_A}{X}(x) \right) dy dx \\
		+ \int_{A_n} \dfrac{\chi_{A_n}u_A}{X} (x) \int_{B_n} R(x,y) \left(  \dfrac{\chi_{A_n}u_A}{X}(y) - \dfrac{\chi_{A_n}u_A}{X}(x) \right) dy dx \\
		+ \int_{B_n} \dfrac{\chi_{A_n}u_A}{X} (x) \int_{A_n} R(x,y) \left(  \dfrac{\chi_{A_n}u_A}{X}(y) - \dfrac{\chi_{A_n}u_A}{X}(x) \right) dy dx \\
		+ \int_{B_n} \dfrac{\chi_{A_n}u_A}{X} (x) \int_{B_n} G(x,y) \left(  \dfrac{\chi_{A_n}u_A}{X}(y) - \dfrac{\chi_{A_n}u_A}{X}(x) \right) dy dx.
		\end{gathered}
		$$
		Recall that $\chi_{A_n}(x) \equiv 0$ in $B_n$, and then, we obtain 
		$$
		\begin{gathered}
		a\left( \dfrac{\chi_{A_n}u_A}{X}, \dfrac{\chi_{A_n}u_A}{X} \right) = 
		\int_{A_n} \dfrac{\chi_{A_n}u_A}{X} (x) \int_{A_n} J(x,y) \left(  \dfrac{\chi_{A_n}u_A}{X}(y) - \dfrac{\chi_{A_n}u_A}{X}(x) \right) dy dx \\
		- \int_{A_n} \dfrac{\chi_{A_n}u_A^2}{{X}^2} (x) \int_{B_n} R(x,y) \, dy dx.  
		\end{gathered}
		$$
		
		Thus, 
		\begin{equation} \label{eq670}
		\begin{gathered}
		\lim_{n\to \infty} a\left( \dfrac{\chi_{A_n}u_A}{X}, \dfrac{\chi_{A_n}u_A}{X} \right) =  
		\int_\Omega \frac{u_A}{X}(x) \int_\Omega J(x,y) \left( X(x) u_A(y) - X(y) u_A(x)  \right) dydx \\
		\qquad \qquad \qquad \qquad \qquad - \int_\Omega \frac{u_A}{X}(x) \int_\Omega R(x,y) (1-X(y)) u_A(x) \, dy dx.
		\end{gathered}
		\end{equation}
		
		We can argue in analogous way to obtain that
		\begin{equation} \label{eq680}
		\begin{gathered}
		\lim_{n\to \infty} a\left( \dfrac{\chi_{B_n}u_B}{1-X}, \dfrac{\chi_{B_n}u_B}{1-X} \right) =  
		- \int_\Omega \frac{u_B}{1-X}(x) \int_\Omega R(x,y) X(y) u_B(x) \, dy dx \\
		\qquad \qquad \qquad + \int_\Omega \frac{u_B}{1-X}(x) \int_\Omega G(x,y) \left( (1-X(x)) u_B(y) - (1-X(y)) u_B(x)  \right)\, dydx.
		\end{gathered}
		\end{equation}

		Next, we see that
		$$
		\begin{gathered}
		a\left( \dfrac{\chi_{A_n}u_A}{X}, \dfrac{\chi_{B_n}u_B}{1-X} \right) 
		= \int_{A_n} \dfrac{\chi_{A_n}u_A}{X}(x) \int_{B_n} R(x,y) \left( \dfrac{\chi_{B_n}u_B}{1-X}(y) - \dfrac{\chi_{B_n}u_B}{1-X}(x) \right) dy dx \\
		\qquad \qquad \qquad \qquad \qquad \qquad 
		+ \int_{B_n} \dfrac{\chi_{A_n}u_A}{X}(x) \int_{A_n} R(x,y) \left( \dfrac{\chi_{B_n}u_B}{1-X}(y) - \dfrac{\chi_{B_n}u_B}{1-X}(x) \right) dy dx \\
		= \int_{A_n} \dfrac{\chi_{A_n}u_A}{X}(x) \int_{B_n} R(x,y) \dfrac{\chi_{B_n}u_B}{1-X}(y) \, dy dx
		\end{gathered}
		$$
		since $\chi_{A_n}(x) \chi_{B_n}(x) \equiv 0$ in $\Omega$. Then,
		\begin{equation} \label{eq700}
		\lim_{n\to\infty} a\left( \dfrac{\chi_{A_n}u_A}{X}, \dfrac{\chi_{B_n}u_B}{1-X} \right)  
		= \int_{\Omega} u_A(x) \int_{\Omega} R(x,y) u_B(y) \, dy dx.
		\end{equation}
		
		Therefore, we can conclude from \eqref{eq630}, \eqref{eq640}, \eqref{eq670}, \eqref{eq680} and \eqref{eq700} that 
		$$
		\begin{array}{l}
		\displaystyle 
		\lim_{n \to \infty} a\left( u_n-\dfrac{\chi_{A_n}u_A}{X}-\dfrac{\chi_{B_n}u_B}{1-X}, u_n-\dfrac{\chi_{A_n}u_A}{X}-\dfrac{\chi_{B_n}u_B}{1-X} \right)   \\[10pt] \displaystyle  =
		- \int_\Omega f u_A \, dx 
		+ \int_\Omega \frac{u_A}{X}(x) \int_\Omega J(x,y) \left( X(x) u_A(y) - X(y) u_A(x)  \right) dydx 
		\\[10pt] \displaystyle
		 \qquad + \int_\Omega \frac{u_A}{X}(x) \int_\Omega R(x,y) (X(x) u_B(y) - (1-X(y)) u_A(x) \, dy dx 
		 \\[10pt] \displaystyle \qquad
		- \int_\Omega f u_B \, dx 
		+ \int_\Omega \frac{u_B}{1-X}(x) \int_\Omega G(x,y) \left( (1-X(x)) u_B(y) - (1-X(y)) u_B(x)  \right) dydx 
		\\[10pt] \displaystyle \qquad
		+ \int_\Omega \frac{u_B}{1-X}(x) \int_\Omega R(x,y) ((1-X(x)) u_A(y) - X(y) u_B(x) \, dy dx 
		\\[10pt] \displaystyle
		= 0
		\end{array}
		$$
		since $u_A$ and $u_B$ satisfy \eqref{limitAn} and \eqref{limitBn} respectively, and $X$ satisfies \eqref{conX}.
		Consequently, we obtain \eqref{n630} from Lemma \ref{lemmaeigenvalue}.
		\end{proof}

		\section{The Dirichlet problem}
		\label{Sect.proofs.Dir}
		\setcounter{equation}{0}
		
	In this section we discuss the proofs of the results concerned to the Dirichlet problem \eqref{1.1.Dir}.
	As for the Neumann case, we start by the analysis of an eigenvalue problem that is needed to obtain existence and
	uniqueness of solutions.

			Now, let us consider the generalized eigenvalue for the Dirichlet problem \eqref{1.1.Dir} which is necessary to show uniform estimates. 
			We introduce 
			$$
			\Phi(A,B,V)w := \int_{A}\int_{B}V(x,y)(w(y)-w(x))^2 dy dx
			$$
			which is set for any open bounded sets $A$ and $B \subset \R^N$, $V \in \mathcal{C}(\R^N,\R)$ and $w \in L^2(\R^N)$.
			
			We consider the following quantity 
			\begin{equation}\label{lambdaN.Dir}
			\lambda_n(\Omega)=\dfrac{1}{2}\inf_{\begin{array}{c}
				u \in W_{Dir}\\u \neq 0
				\end{array}} 
			\dfrac{\Phi(A_n,A_n,J)u + 2\Phi(A_n,B_n,R)u + \Phi(B_n,B_n,G)u}{||u||^2_{L^2(\Omega)}}
			\end{equation}
			where 
			$W_{Dir}$ is set in \eqref{SetW_solutions_Dir}. We have the following:
			
			\begin{lemma}\label{lemmaeigenvalue.Dir}
				Let $\{\lambda_n(\Omega) \}$ be the family of values given by \eqref{lambdaN.Dir}. 
				Then, there exists a positive constant $c$ such that
				$$
				\lambda_n(\Omega) > c,\quad\forall n\geq 1.
				$$
			\end{lemma}
			\begin{proof}
				 We can argue as in the proof of Lemma \ref{lemmaeigenvalue} obtaining the result as a consequence of \cite[Proposition 2.3]{ElLibro}.
			\end{proof}
			
			Now, we can see the existence and uniqueness of the solutions to the Dirichlet problem \eqref{1.1.Dir}. 
			\begin{theorem} \label{EUSD}
				Let $W_{Dir}\subset L^2(\R^N)$ be the closed subspace defined in \eqref{SetW_solutions_Dir} and suppose that the condition $(\mathbf{H})$ holds for the non-singular kernels $J,\,R$ and $G$. 
				Then, for each $f\in W_{Dir}$, there exists a unique $u\in W_{Dir}$ satisfying \eqref{1.1.Dir} and being the minimizer of the functional \eqref{1.1.energy.Dir}.
			\end{theorem}
			\begin{proof}
			The proof is quite similar to that one given for the Theorem \ref{EUN}. One can see that it is a consequence of Lemma \ref{lambdaN.Dir} applied to the bilinear form $d: W_{Dir} \times W_{Dir} \mapsto \R$ defined by 
		\begin{equation}  \label{bforma.99}
		\begin{array}{rl}
		d(u,v) = & \displaystyle \int_{A_n} v(x) \int_{A_n} J (x,y) (u (y) - u (x)) dy dx 
		%\\[10pt]
		%& \displaystyle 
		+ \int_{A_n} v(x) \int_{B_n} R (x,y) (u (y) - u (x)) dy dx
		\\[10pt]
		& \displaystyle 
		+ \int_{B_n} v(x) \int_{A_n} R (x,y) (u (y) - u (x)) dy dx
		%\\[10pt]
		%& \displaystyle 
		+ \int_{B_n} v(x) \int_{B_n} G (x,y) (u (y) - u (x)) dy dx.
		\end{array}
		\end{equation}
			\end{proof}

		\subsection{Proof of Theorem \ref{theo1.intro.dir}} 
		
			\begin{proof}[Proof of Theorem 
			\ref{theo1.intro.dir}]
				The proof is analogous to that one given to Theorem \ref{theo1.intro}. We briefly describe the steps. First, let us write the variational formulation of \eqref{1.1.Dir}: for all $\varphi\in W_{Dir}$ and $n\geq1$ we have 
				\begin{equation} \label{variationalproblem.Dir}
				\begin{array}{rl}
				& \displaystyle \int_{\Omega} \varphi f(x)dx = \displaystyle \int_{A_n} \varphi(x) \int_{A_n} J (x,y) (u_n (y) - u_n (x)) dy dx 
				+ \int_{A_n} \varphi(x) \int_{B_n} R (x,y) (u_n (y) - u_n (x)) dy dx
				\\[10pt]
				& \displaystyle 
				\qquad\qquad\qquad + \int_{B_n} \varphi(x) \int_{A_n} R (x,y) (u_n (y) - u_n (x)) dy dx
				+ \int_{B_n} \varphi(x) \int_{B_n} G (x,y) (u_n (y) - u_n (x)) dy dx.
				\end{array}
				\end{equation}

				{\textbf{Uniform bounds.}}
				
				It follows from \eqref{variationalproblem.Dir} taking $\varphi=u_n$ and Lemma \ref{lemmaeigenvalue.Dir} shown below that there exists $c>0$ such that
				$
				\|u_n\|_{L^2(\Omega)}\leq c
				$
				for all $n \geq 1$.
				
				{\textbf{Limit of equation \eqref{variationalproblem.Dir} as $n\to \infty$.}} 
				
				Using the fact that the solutions $u_n$ are uniformly bounded, we have that $\chi_{A_n}u_n$ and $\chi_{B_n}u_n$ are uniformly bounded. Therefore, there are $u_A,u_B\in W_{Dir}$ such that, up to subsequences,
				\begin{equation}\label{1099}
				\chi_{A_n}u_n\rightharpoonup u_A\quad\textrm{and}\quad \chi_{B_n}u_n\rightharpoonup u_B
				\end{equation}
				weakly in $W_{Dir}$. Moreover, as in \eqref{convkernel}, we have for $V = J$, $R$ or $G$ that
				\begin{equation}\label{1103}
		\begin{gathered}
		\int_{\R^N}\chi_{A_n}(y) \, V(\cdot -y) dy\to \int_{\R^N} X(y) \, V(\cdot-y) dy \quad \textrm{ and } \\
		\int_{\R^N}\chi_{B_n}(y) \, V(\cdot-y) dy\to \int_{\R^N} (1-X)(y) \, V(\cdot-y) dy
		\end{gathered}
		\end{equation}
		strongly in $L^2(\Omega)$ as $n \to \infty$.

				Next, let us rewrite \eqref{variationalproblem.Dir} as follows:
				\begin{equation*} \label{variationaproblem.Dir.ntoinfty} 
				\begin{array}{rl}
			 \displaystyle \int_{\Omega} \varphi f(x)dx  =&  \displaystyle \int_\Omega\chi_{A_n}(x)\varphi(x) \int_{\R^N}J (x,y)\chi_{A_n}(y)  u_n (y)dy dx -\int_{\R^N}\chi_{A_n}(x) \varphi(x)u_n(x)\int_{\R^N}\chi_{A_n}(y)J(x,y)dy dx\\[10pt]
				& + \displaystyle 
				 \int_{\R^N}\chi_{A_n}(x) \varphi(x) \int_{\R^N} R (x,y)\chi_{B_n}(y) u_n (y) dy dx-\int_{\R^N}\chi_{A_n} \varphi(x)u_n(x)\int_{\R^N}\chi_{B_n}(y)R(x,y)dy dx
				\\[10pt]
				&+ \displaystyle 
				 \int_{\R^N}\chi_{B_n}(x) \varphi(x) \int_{\R^N} R (x,y)\chi_{A_n}(y) u_n (y) dy dx-\int_{\R^N}\chi_{B_n} \varphi(x)u_n(x)\int_{\R^N}\chi_{A_n}(y)R(x,y)dy dx
				\\[10pt]
				&+ \displaystyle \int_\Omega\chi_{B_n}(x)\varphi(x) \int_{\R^N}G (x,y)\chi_{B_n}(y)  u_n (y)dy dx -\int_{\R^N}\chi_{B_n}(x) \varphi(x)u_n(x)\int_{\R^N}\chi_{B_n}(y)G(x,y)dy dx
				\end{array}
				\end{equation*}
				Hence, with the convergences \eqref{1099} and \eqref{1103} in hand, one can proceed as in \eqref{eq660} to get
				\begin{equation*} 
				\begin{array}{l}
				\displaystyle\int_\Omega f(x)\varphi(x) dx= \displaystyle\int_{\R^N}\varphi(x)\int_{\R^N}J(x,y)\left(X(x)u_A(y)-X(y)u_A(x) \right)dydx\\[10pt]
				\qquad +\displaystyle\int_{\R^N}\varphi(x)\int_{\R^N} R(x,y)\left[(1-X(x))u_A(y)-(1-X(y))u_A(x)
				+X(x)u_B(y)-X(y)u_B(x)\right] dydx\\[10pt]
				\qquad +\displaystyle\int_{\R^N}\varphi(x)\int_{\R^N} G(x,y)\left[ (1-X(x))u_B(y)- (1-X(y))u_B(x) \right] dydx.
				\end{array}
				\end{equation*}

				{\textbf{Limit for test functions $\chi_{A_n}\varphi$ and $\chi_{B_n}\varphi$.}}
				
				One can proceed as in \eqref{proofpart3} obtaining  
				\begin{equation*}\label{eqlimitA.Dir}
				\begin{array}{rl}
				\displaystyle \int_\Omega X \varphi f(x)dx = & \displaystyle\int_{\R^N}\varphi(x)\int_{\R^N}J(x,y)\left(X(x)u_A(y)-X(y)u_A(x) \right)dydx \\[10pt]
				& \displaystyle 
				+ \int_{\R^N}  \varphi(x) \int_{\R^N} R (x,y) \left[X(x)u_B (y)-u_A(x)(1-X(y)) \right] dy dx,
				\end{array}
				\end{equation*}
				and 
				\begin{equation*}  
				\begin{array}{rl}
				\displaystyle \int_\Omega (1-X(x))\varphi f(x)dx  = & 
				\displaystyle 
				\int_{\R^N}\varphi(x) \int_{\R^N} R (x,y) \left[u_A (y)(1-X(x)) -u_B(x)X(y) \right]dy dx 
				\\[10pt]
				& \displaystyle 
				+ \int_{\R^N}\varphi(x) \int_{\R^N} G (x,y)\left[ u_B (y)(1-X(x))-u_B(x)(1-X(y))\right]  dy dx
				\end{array}
				\end{equation*}
				for all $\varphi\in W_{Dir}$.
			\end{proof}

			Here we prove Corollary \ref{corExD}. We deal with the particular partition of the $\R^N$ 
			in which take $\chi A_n \to 1$ as $n\to \infty$. 
			
			\begin{proof}[Proof of Corollary \ref{corExD}]				
Since $\chi A_n \to 1$ as $n\to \infty$ we have that $X\equiv 1$ in $\mathbb{R}^N$, therefore
from equation \eqref{limitBn_dir} we get that
$$
\chi_{B_n } u_n \to u_B=0,
$$
and then from \eqref{limitAn_dir} we obtain the limit equation 
				\begin{equation*}
				\displaystyle\int_{\R^N}f(x)\varphi(x)dx=\displaystyle\int_{\R^N}\varphi(x)\int_{\R^N}J(x,y)u_A(y) dy dx-\displaystyle\int_{\R^N}\varphi(x)u_A(x)\int_{\R^N}J(x,y)dy dx
				\end{equation*}
				
				Finally, one can proceed as in the proof of Corollary \ref{corX=0} to get
				$$
				\|u_n-u_A\|_{L^2(\Omega)}^2\to 0, \quad \textrm{ as } n \to \infty.
				$$
			\end{proof}

\subsection{Proof of Theorem \ref{Corrector.Dirichlet}}

The proof of the result concerning correctors for the Dirichlet problem 
follows the same steps as in the proof for the Neumann case, Theorem \ref{Corrector.Neumann}, and hence it is left to the reader.

\vspace{0.5 cm}

{\bf Acknowledgements.} 
The first and last authors (MC and JDR) are partially supported by 
CONICET grant PIP GI No 11220150100036CO
(Argentina), UBACyT grant 20020160100155BA (Argentina), 
Project MTM2015-70227-P (Spain).

The third author (MCP) has been partially supported by CNPq 303253/2017-7 (Brazil).

	\noindent\textbf{Addresses:}

{Monia Capanna and Julio D. Rossi
\hfill\break\indent
CONICET and Departamento  de Matem{\'a}tica, FCEyN,\hfill\break\indent 
Universidad de Buenos Aires, 
\hfill\break\indent  Ciudad Universitaria, Pabellon I, (1428).
Buenos Aires, Argentina.}
\hfill\break\indent
{{\tt moniacapanna@gmail.com, jrossi@dm.uba.ar}}

{Jean C. Nakasato and Marcone C. Pereira
		\hfill\break\indent Dpto. de Matem{\'a}tica Aplicada, IME,
		Universidade de S\~ao Paulo, \hfill\break\indent Rua do Mat\~ao 1010, 
		S\~ao Paulo - SP, Brazil. } \hfill\break\indent {{\tt nakasato@ime.usp.br, marcone@ime.usp.br} \hfill\break\indent {\it
			Web page: }{\tt www.ime.usp.br/$\sim$nakasato, www.ime.usp.br/$\sim$marcone}}

%{Marcone C. Pereira
%		\hfill\break\indent Dpto. de Matem{\'a}tica Aplicada, IME,
%		Universidade de S\~ao Paulo, \hfill\break\indent Rua do Mat\~ao 1010, 
%		S\~ao Paulo - SP, Brazil. } \hfill\break\indent {{\tt marcone@ime.usp.br} \hfill\break\indent {\it
%			Web page: }{\tt www.ime.usp.br/$\sim$marcone}}

\end{document}